\documentclass[10pt]{amsart}
\usepackage{amssymb}
\usepackage{amsmath}
\usepackage{amsfonts}
\usepackage{mathrsfs}
\usepackage{psfrag}
\usepackage{verbatim}
\usepackage{color}
\usepackage{bm}
\usepackage[usenames,dvipsnames]{xcolor}
\usepackage{enumerate}
\usepackage{dsfont}
\newcommand{\ind}{\mathds{1}}

\DeclareGraphicsExtensions{}
\theoremstyle{plain}
\newtheorem{theorem}{Theorem}[section]
\newtheorem{corollary}[theorem]{Corollary}

\newtheorem{lemma}[theorem]{Lemma}
\newtheorem{proposition}[theorem]{Proposition}

\theoremstyle{definition}
\newtheorem{assumption}[theorem]{Assumption}
\newtheorem{definition}[theorem]{Definition}

\theoremstyle{remark}
\newtheorem{remark}[theorem]{Remark}

\newcommand{\eps}{\varepsilon}

\newcommand{\la}{\left\langle}
\newcommand{\ra}{\right\rangle}

\newcommand{\BR}{\mathbb{R}}
\newcommand{\BP}{\mathbb{P}}
\newcommand{\BE}{\mathbb{E}}
\newcommand{\BZ}{\mathbb{Z}}

\newcommand{\BT}{\mathbb{T}}

\newcommand{\II}{\mathscr{I}} % set of discrete states (configurations)
\newcommand{\mm}{\mathsf{m}} % cardinality of state space of discrete component
\newcommand{\dd}{\mathsf{d}} % dimension of state space of continuous component
\newcommand{\rr}{\mathsf{r}} % dimension of Brownian motion
\newcommand{\per}{\textsf{per}} % for periodic
\newcommand{\filt}{\mathscr{F}} 

\newcommand{\Borel}{\mathscr{B}}
\newcommand{\gen}{\mathscr{L}}

\newcommand{\HH}{\mathbb{H}} % function space for Lax-Milgram
\newcommand{\LL}{\mathbb{L}} % product L^2 space
\newcommand{\loc}{\mathsf{loc}} % for subscript local
\newcommand{\mi}{r} % for multiindex
\newcommand{\sfdiv}{\mathsf{div}} % symbol for divergence
\newcommand{\blf}{\mathfrak{B}} % bilinear form for Lax-Milgram
\newcommand{\NN}{\mathsf{N}} % N for Poisson random measure
\newcommand{\Leb}{\mathsf{Leb}} % Lebesgue measure
\newcommand{\ttM}{\mathtt{M}} % martingales stemming from the Ito formula 
\newcommand{\scQ}{\mathscr{Q}} % operateur carre du champ
\newcommand{\ttC}{\mathtt{C}} % related to limiting covariance matrix
\newcommand{\scC}{\mathscr{C}} % space of continuous paths
\newcommand{\scD}{\mathscr{D}} % space of cadlag paths
\newcommand{\nullspace}{\mathscr{N}}
\newcommand{\Sk}{\mathsf{Sk}}
\newcommand{\ttJ}{\mathtt{J}} % intervals for h function in PRM SDE

\newcommand{\sfmin}{\mathsf{min}}

\begin{document}
\title{Periodic Homogenization for Switching Diffusions}
\author{Chetan D. Pahlajani}
\address{Department of Mathematics\\ Indian Institute of Technology Gandhinagar}
\email{cdpahlajani@iitgn.ac.in} 
\date{\today.}
\thanks{The author's research was supported by ANRF project number MTR/2023/000545. Thanks are due to Kaustubh Rane for several helpful discussions on molecular motors, which provided part of the motivation for the present work, and to Harsha Hutridurga for discussing homogenization techniques, especially the results in \cite{AllaireHutridurga-DCDSB-2015}. Finally, the author would like to acknowledge the school on ``Quantitative Theory of Homogenization" held at IIT Bombay in February 2025 for the valuable lectures and fruitful conversations.}

\maketitle

\begin{abstract}
In the present work, we explore homogenization techniques for a class of switching diffusion processes whose drift and diffusion coefficients, and jump intensities are smooth, spatially periodic functions; we assume full coupling between the continuous and discrete components of the state. Under the assumptions of uniform ellipticity of the diffusion matrices and irreducibility of the matrix of switching intensities, we explore the large-scale long-time behavior of the process under a diffusive scaling. Our main result characterizes the limiting fluctuations of the rescaled continuous component about a constant velocity drift by an effective Brownian motion with explicitly computable covariance matrix. In the process of extending classical periodic homogenization techniques for diffusions to the case of switching diffusions, our main quantitative finding is the computation of an extra contribution to the limiting diffusivity stemming from the switching.
\end{abstract}

\section{Introduction}\label{S:Introduction}
% Paragraph 1: About switching diffusion processes.
%\noindent \color{violet} (\textsf{Paragraph 1: About switching diffusion processes.}) \color{black}\\
Mathematical models for several real-world problems are naturally formulated in terms of \textit{switching Markov processes}; these are two-component stochastic processes $(X_t,I_t)$ with a continuous component $X_t$ evolving in some subset of Euclidean space and a discrete component $I_t$ taking values in a finite set. Here, the evolution of $X_t$ is governed by one of a family of Markovian generators based on the current (instantaneous) value of $I_t$, while transitions among different values of $I_t$ are governed by a jump stochastic process with intensities depending on $X_t$. Of particular interest are scenarios where the dynamics of $X_t$ are governed by a stochastic differential equation ({\sc sde}). Such \textit{switching diffusion processes} \cite{YinZhu-book}, also called \textit{diffusion-transmutation processes} \cite{FreidlinLee}, arise in numerous applications spanning the spectrum from finance and manufacturing through materials science and molecular biology; see \cite{GAM_SICON1992,YinZhu-book,TongMajda,PeletierSchlottke_EJP2024} and references therein.

% Paragraph 2: About homogenization
%\noindent \color{violet} (\textsf{Paragraph 2: About homogenization.}) \color{black}\\
In a number of applications, one encounters stochastic processes or differential equations whose coefficients exhibit spatial periodicity on a scale much smaller than the macroscopic scale of interest. An important analytical tool in such situations is the use of \textit{homogenization} \cite{BLP,PavliotisStuart} techniques to average out the rapid microscopic variations and obtain, in the limit of increasing scale separation, simpler (and often more computationally tractable) approximations. Homogenization results for stochastic processes whose governing {\sc sde} or generators exhibit spatially periodic coefficients have been extensively explored; we mention \cite{BLP,Rabi_AoP1985,HairerPavliotis_JStatPhys2004,CCKW_AoP2021,KrempPerkowski,Sandric2024} by way of a representative, albeit incomplete, list. While large classes of stochastic processes, including those with L\'evy-type generators \cite{CCKW_AoP2021,KrempPerkowski}, have been studied, with the exception of \cite{Sandric2024}, little attention seems to have been devoted to the case of switching diffusions.  
  
% Paragraph 3: What we do here
%\noindent \color{violet} (\textsf{Paragraph 3: What we do here.}) \color{black}\\
The goal of the present work is to explore homogenization for a class of switching diffusion processes whose drift and diffusion coefficients, together with switching intensities, are spatially periodic smooth functions. Starting from this (microscopic) process $(X_t,I_t)$, we use a diffusive scaling to understand the large-scale, long-time (macroscopic) dynamics of the process $(X^\eps_t,I^\eps_t) \triangleq (\eps X_{t/\eps^2}, I_{t/\eps^2})$; here, $0<\eps \ll 1$. Under the assumptions of uniform ellipticity of the diffusion coefficients and irreducibility of the matrix of jump intensities, our main result takes the form $X^\eps_t - \bar{b}\thinspace t/\eps \Longrightarrow \sqrt{\ttC}B_t$, thus characterizing the limiting fluctuations (in the sense of convergence in distribution) of $X^\eps_t$ about a constant velocity drift $\bar{b}/\eps$ by $\sqrt{\ttC}B_t$ which is an effective Brownian motion with explicitly computable covariance matrix $\ttC$. The expression for this limiting covariance matrix involves a term stemming from the diffusive dynamics, similar to that obtained in the classical theory \cite{BLP,Rabi_AoP1985}, and a new term which quantifies the contribution of the \textit{switching} dynamics to the limiting diffusivity; see Remark \ref{R:Interpretation}. Computation of this new term is arguably our main quantitative finding.   

% Paragraph 4: Novelty
%\noindent \color{violet} (\textsf{Paragraph 4: Novelty.}) \color{black}\\
As noted, the present work can be viewed as an extension of the classical periodic homogenization results for diffusions \cite{BLP,Rabi_AoP1985} to the case of switching diffusions. In terms of the analysis, this translates to working not with a single partial differential equation ({\sc pde}), but rather a \textit{system} of weakly coupled elliptic {\sc pde} while solving the \textit{cell problem}. The other novel technical feature is an extra jump martingale term in the Ito formula which arises due to the switching. In fact, it is the asymptotic behavior of this quantity that provides the additional switching contribution to the limiting diffusivity alluded to above. 

% Paragraph 5: Relation to existing work
%\noindent \color{violet} (\textsf{Paragraph 5: Relation to existing work.}) \color{black}\\
To help place this work in context and better elucidate its connection with the literature, a few comments are in order. In \cite{PeletierSchlottke_EJP2024}, the authors consider a very similar class of switching Markov processes and obtain, in an Eulerian scaling, a large deviation principle governing the probabilities of rare events. Our work complements these results by working in a diffusive scaling and quantifying typical fluctuations by means of a central limit theorem ({\sc clt}); see Remark \ref{R:Comparison}. While the paper \cite{Sandric2024} explores homogenization for a class of weakly coupled parabolic {\sc pde} using the associated switching diffusion (somewhat similar to ours) as a tool, the analysis is restricted to the case of constant switching intensities and takes a different approach. More importantly, the results in \cite{Sandric2024} do not appear to account for the extra contribution to the limiting diffusivity stemming from the switching. Finally, a slight change of perspective shows that our setting is closely related to that in \cite{AllaireHutridurga-DCDSB-2015}. Indeed, the latter studies, using analysis and {\sc pde} techniques, homogenization for a system of parabolic {\sc pde} which can be seen to resemble the backward Kolmogorov equations for our switching diffusions. In the present work, we assume greater regularity (smoothness) of coefficients compared to  
\cite{AllaireHutridurga-DCDSB-2015}, but use (mostly) probabilistic techniques to obtain a limit theorem regarding weak convergence of probability laws on function space.

% Paragraph 6: Organization of the paper, notational conventions, comments.
%\noindent \color{violet} (\textsf{Paragraph 6: Organization of the paper, notational conventions, comments.}) \color{black}\\
The rest of the paper is organized as follows. Section \ref{S:PFMR} starts with a formulation of our problem of interest, together with the appropriate assumptions. A key component in our analysis is the Fredholm alternative, which is stated in Proposition \ref{P:Fredholm}, and immediately yields a solution to the cell problem in Corollary \ref{C:cell}. Using these tools, we state our main result in Theorem \ref{T:Main}. Section \ref{S:Homogenization} is devoted to the proof of Theorem \ref{T:Main}, taking Proposition \ref{P:Fredholm} and Corollary \ref{C:cell} as the starting point. The calculations involve use of the Ito formula, a proof of ergodicity, and the use of the martingale {\sc clt}. Finally, the proof of the Fredholm alternative (Proposition \ref{P:Fredholm}) is provided in Section \ref{S:Fredholm}. Section \ref{S:Conclusions} concludes the paper with some observations and possible directions for future research.

\section{Problem Formulation and Main Result}\label{S:PFMR}
Fix positive integers $\dd,\mm \ge 1$. Consider a two-component stochastic process $(X_t,I_t)$ with continuous component $X_t$ taking values in $\BR^\dd$, and discrete component $I_t$ taking values in $\II \triangleq \{1,2,\dots,\mm\}$, which evolves according to
\begin{equation}\label{E:sw-sde-micro}
dX_t = b(X_t,I_t) \thinspace dt + \sigma(X_t,I_t) \thinspace dW_t\\
\end{equation}
\begin{equation}\label{E:jump-micro}
\BP\left(I_{t+\Delta}=\beta| \thinspace I_t=\alpha, (X_s,I_s), 0 \le s \le t\right) =q_{\alpha\beta}(X_t)\Delta + o(\Delta), \quad \text{as $\Delta \searrow 0$, $\alpha \neq \beta$,}
\end{equation}
where $W_t$ is an $\rr$-dimensional Brownian motion, the coefficients $b(\cdot,\alpha):\BR^\dd \to \BR^\dd$, $\sigma(\cdot,\alpha):\BR^\dd \to \BR^{\dd\times \rr}$, $1 \le \alpha \le \mm$, and  transition intensities $q_{\alpha\beta}:\BR^\dd \to [0,\infty)$, $\alpha,\beta \in \II$, $\alpha \neq \beta$, are sufficiently regular; see Assumptions \ref{A:DD}--\ref{A:Periodicity} below. Such switching diffusion processes are used to model situations where a system of interest with diffusive dynamics may be operating in one of $\mm$ ``modes" or configurations, with state-dependent intensities governing the stochastic switching among modes. 

\begin{assumption}[Drift and diffusion coefficients]\label{A:DD}
For each $\alpha \in \{1,\dots,\mm\}$, the functions $b(\cdot,\alpha): \BR^\dd \to \BR^\dd$ and $\sigma(\cdot,\alpha): \BR^\dd \to \BR^{\dd \times \rr}$ are smooth.  
Further, the diffusion matrices $a(x,\alpha)\triangleq \sigma(x,\alpha)\sigma(x,\alpha)^\intercal$ are uniformly elliptic, i.e., there exists a constant $a_\sfmin>0$ such that 
\begin{equation}\label{E:ellipticity}
 \left(a(x,\alpha)\xi, \xi \right)_{\BR^\dd} \ge a_\sfmin|\xi|^2 \qquad \text{for all $x,\xi \in \BR^\dd$, $1 \le \alpha \le \mm$.}
\end{equation}
\end{assumption}

\begin{assumption}[Switching intensities]\label{A:SI}
The family of functions $\{q_{\alpha\beta}(x)\}_{1 \le \alpha,\beta \le \mm}$ satisfies the following:
\begin{enumerate}[(i)]
\item For each $1 \le \alpha, \beta \le \mm$, $\beta \neq \alpha$, the function $q_{\alpha\beta}:\BR^\dd \to [0,\infty)$ is smooth and bounded. Thus, there exists $\bar{q}>0$ such that $0 \le q_{\alpha\beta}(x) \le \bar{q}$ for all $\beta \neq \alpha$, $x \in \BR^\dd$. 
\item For each $\alpha$, set $q_{\alpha\alpha}(x) \triangleq -\sum_{\substack{\beta =1\\\beta \neq \alpha}}^\mm q_{\alpha\beta}(x) \le 0$; we note that for each $\alpha \in \II$, we have $\sum_{\beta=1}^\mm q_{\alpha\beta}(x)=0$ for all $x \in \BR^\dd$.
\item The matrix $Q(x)=[q_{\alpha\beta}(x)]_{1 \le \alpha,\beta \le \mm}$ is irreducible for each $x \in \BR^\dd$. Thus, it is not possible to partition $\II$ into nonempty disjoint sets $\mathcal{B}$ and $\mathcal{B}^\prime$ such that $q_{\alpha\beta}(x)=0$ for all $\alpha \in \mathcal{B}$, $\beta \in \mathcal{B}^\prime$. 
\end{enumerate}
\end{assumption}

\begin{assumption}[Periodicity]\label{A:Periodicity}
The drift and diffusion coefficients, and the switching intensities are $1$-periodic in $x$ for each fixed $\alpha \in \II$, i.e., 
\begin{equation*}
b(x+\nu,\alpha)=b(x,\alpha), \quad \sigma(x+\nu,\alpha)=\sigma(x,\alpha), \quad q_{\alpha\beta}(x+\nu)=q_{\alpha\beta}(x) \qquad \text{for all $x \in \BR^\dd$, $\alpha, \beta \in \II$, $\nu \in \BZ^\dd$.}
\end{equation*}
\end{assumption}

To understand the behavior of the process $(X_t,I_t)$ at large spatial and temporal scales, we introduce a small parameter $0 < \eps \ll 1$, and consider the rescaled stochastic process $(X^\eps_t,I^\eps_t) \triangleq (\eps X_{t/\eps^2}, I_{t/\eps^2})$. For each (fixed) $\eps \in (0,1)$, the equations \eqref{E:sw-sde-micro} and \eqref{E:jump-micro} now become
\begin{equation}\label{E:sw-sde-macro}
dX^\eps_t = \frac{1}{\eps} b\left(\frac{X^\eps_t}{\eps},I^\eps_t\right) \thinspace dt + \sigma\left(\frac{X^\eps_t}{\eps},I^\eps_t\right) \thinspace dW_t\\
\end{equation}
\begin{equation}\label{E:jump-macro}
\BP\left(I^\eps_{t+\Delta}=\beta| \thinspace I^\eps_t=\alpha, (X^\eps_s,I^\eps_s), 0 \le s \le t\right) =\frac{1}{\eps^2} q_{\alpha\beta}\left(\frac{X^\eps_t}{\eps}\right)\Delta + o(\Delta), \quad \text{as $\Delta \searrow 0$, $\alpha \neq \beta$.}
\end{equation}
Our main result, formulated precisely in Theorem \ref{T:Main}, asserts that as $\eps \searrow 0$, the process $X^\eps_t$ can be approximated (in a suitable sense) by $\overline{b}\thinspace t/\eps + \sqrt{\ttC} B_t$, where $\overline{b}$ is an effective drift obtained by averaging $b(x,\alpha)$ with respect to a suitable invariant measure, $B_t$ is a $\dd$-dimensional Brownian motion, and $\ttC$ is an explicitly computable covariance matrix which is obtained by \textit{homogenization}, i.e., averaging over the fast spatially periodic variations in \eqref{E:sw-sde-macro}, \eqref{E:jump-macro}.

We start by writing both the microscopic process $(X_t,I_t)$ solving \eqref{E:sw-sde-micro}, \eqref{E:jump-micro} and the macroscopic process $(X^\eps_t,I^\eps_t)$ solving \eqref{E:sw-sde-macro}, \eqref{E:jump-macro} as solutions to {\sc sde} involving \textit{Poisson random measures} ({\sc prm}); see \cite{IkedaWatanabe} for a detailed rigorous treatment of {\sc prm}, including the concepts and tools used in the sequel. Let $(\Omega,\filt,\BP)$ be a probability space equipped with the filtration $\{\filt_t:t \ge 0\}$, on which is defined an $\rr$-dimensional $\filt_t$-Brownian motion $W_t$. Assume that this setup is rich enough to also support an independent (of $W_t$) $\filt_t$-\textit{Poisson point process} ({\sc ppp}) $p$ on $\BR$ with intensity measure $dt \times \Leb(dz)$, where $\Leb$ denotes Lebesgue measure on $\BR$. Let $\NN_p(dt,dz)$ be the (random) counting measure on $(0,\infty)\times \BR$ induced by $p$. By assumption, $\NN_p(dt,dz)$ is a {\sc prm} on $(0,\infty)\times \BR$ with $\BE[\NN_p(dt,dz)]=dt \times \Leb(dz)$. For $t \ge 0$, $U \in \Borel(\BR)$, define the \textit{compensator} of $\NN_p(t,U)$ by 
\begin{equation*}
\widehat{\NN}_p(t,U) \triangleq \BE[\NN_p(t,U)] = t \thinspace \Leb(U),
\end{equation*}
where the last equality follows from the assumptions regarding $p$. Note that for any $U \in \Borel(\BR)$ with $\Leb(U)<\infty$, the process $\{\widetilde{\NN}_p(t,U):t \ge 0\}$ defined by 
\begin{equation}\label{E:cprm}
\widetilde{\NN}_p(t,U) \triangleq \NN_p(t,U) - \widehat{\NN}_p(t,U) \qquad \text{is an $\filt_t$-martingale.}
\end{equation}

The algorithm for constructing our processes of interest as solutions to {\sc sde} involving {\sc prm} is based on \cite{GAM_SICON1992}, \cite{YinZhu-book}. We start with the microscopic process. For $x \in \BR^\dd$, $1 \le \alpha,\beta \le \mm$, $\alpha \neq \beta$, let $\ttJ_{\alpha\beta}(x)$ be the consecutive lexicographically ordered left-closed right-open intervals on $\BR$ with $\ttJ_{\alpha\beta}(x)$ having length $q_{\alpha\beta}(x)$. Now define a function $h:\BR^\dd \times \II \times \BR \to \BR$ by
%\begin{equation*}
$h(x,\alpha,z) \triangleq \sum_{\beta=1}^\mm (\beta-\alpha) \ind_{\ttJ_{\alpha\beta}(x)}(z)$.
%\end{equation*}
The process $(X_t,I_t)$ can be obtained as a solution of the {\sc sde}
\begin{equation*}\label{E:sde-prm-micro}
\begin{aligned}
dX_t &= b(X_t,I_t) \thinspace dt + \sigma(X_t,I_t) \thinspace dW_t\\
dI_t &= \int_{\BR} h(X_{t-},I_{t-},z) \thinspace \NN_p(dt,dz),\\
(X_0,I_0) &= (x_0,\alpha_0) \in \BR^\dd \times \II.
\end{aligned}
\end{equation*}
In the macroscopic picture, the transition intensities $q_{\alpha\beta}^\eps(x)$ are given by
\begin{equation}\label{E:sw-int-macro}
q_{\alpha\beta}^\eps(x) \triangleq \frac{1}{\eps^2}q_{\alpha\beta}\left(\frac{x}{\eps}\right) \qquad \text{for $x \in \BR^\dd$, $1 \le \alpha, \beta \le \mm$, $\beta \neq \alpha$, and $\eps \in (0,1)$.}
\end{equation}
We let $\ttJ_{\alpha\beta}^\eps(x)$, $\beta \neq \alpha$, be the consecutive lexicographically ordered left-closed right-open intervals on $\BR$ with $\ttJ_{\alpha\beta}^\eps(x)$ having length $q_{\alpha\beta}^\eps(x)$ and define the function $h^\eps:\BR^\dd \times \II \times \BR \to \BR$ by
\begin{equation}\label{E:h-eps}
h^\eps(x,\alpha,z) \triangleq \sum_{\beta=1}^\mm (\beta-\alpha) \ind_{\ttJ_{\alpha\beta}^\eps(x)}(z).
\end{equation}
The process $(X^\eps_t,I^\eps_t)$ can now be obtained as a solution of the {\sc sde}
\begin{equation}\label{E:sde-prm-macro}
\begin{aligned}
dX^\eps_t &= \frac{1}{\eps} b\left(\frac{X^\eps_t}{\eps},I^\eps_t\right) \thinspace dt + \sigma\left(\frac{X^\eps_t}{\eps},I^\eps_t\right) \thinspace dW_t\\
dI^\eps_t &= \int_{\BR} h^\eps(X^\eps_{t-},I^\eps_{t-},z) \thinspace \NN_p(dt,dz),\\
(X^\eps_0,I^\eps_0) &= (\eps x_0,\alpha_0).
\end{aligned}
\end{equation}

\begin{remark}
Note that the expressions defining the quantities $h(x,\alpha,z)$ and $h^\eps(x,\alpha,z)$ (effectively) only involve sums over $\beta \neq \alpha$, regardless of how the sets $\ttJ_{\alpha\alpha}(x)$, $\ttJ^\eps_{\alpha\alpha}(x)$ are defined. Nonetheless, we will take $\ttJ_{\alpha\alpha}(x)=\ttJ^\eps_{\alpha\alpha}(x)=\emptyset$ for all $1 \le \alpha \le \mm$, $x \in \BR^\dd$.
\end{remark}

We next introduce some operators which will play a key role in the analysis. 
For $f:\BR^\dd \times \II \to \BR$ such that $f(\cdot,\alpha):\BR^\dd \to \BR$ is $C^2$ for every $\alpha \in \II$, let
\begin{equation}\label{E:gen-micro}
\begin{aligned}
\gen f(x,\alpha) &\triangleq \sum_{j=1}^\dd b_j(x,\alpha)\frac{\partial f}{\partial x_j}(x,\alpha) + \frac{1}{2}\sum_{j,k=1}^\dd a_{jk}(x,\alpha) \frac{\partial^2 f}{\partial x_j \partial x_k}(x,\alpha) + \sum_{\substack{\beta \in \II\\\beta \neq \alpha}} q_{\alpha\beta}(x) \left[f(x,\beta)-f(x,\alpha)\right],\\
& \text{where $a(x,\alpha) \triangleq \sigma(x,\alpha)\sigma(x,\alpha)^\intercal$.}
\end{aligned}
\end{equation}
The operator $\gen$ is the \textit{generator} of the process $(X_t,I_t)$ and has \textit{adjoint} $\gen^*$ whose action on functions $f:\BR^\dd \times \II \to \BR$ is given by
\begin{equation}\label{E:adjoint}
\begin{aligned}
\gen^* f(x,\alpha) &\triangleq -\sum_{j=1}^\dd \frac{\partial}{\partial x_j}\left(b_j(x,\alpha) f(x,\alpha)\right) + \frac{1}{2}\sum_{j,k=1}^\dd  \frac{\partial^2}{\partial x_j \partial x_k}\left(a_{jk}(x,\alpha) f(x,\alpha)\right) + \sum_{\substack{\beta \in \II\\\beta \neq \alpha}} q^*_{\alpha\beta}(x) \left[f(x,\beta)-f(x,\alpha)\right]\\
& \text{where $q^*_{\alpha\beta}(x) \triangleq q_{\beta\alpha}(x)$ for $\alpha,\beta \in \II$, $x \in \BR^\dd$.}
\end{aligned}
\end{equation}

\begin{remark}\label{R:jump-gen}
It is helpful to note that the difference term on the right-hand side of \eqref{E:gen-micro} admits more than one representation. 
To see this, let $f:\BR^\dd \times \II \to \BR$ with $f(\cdot,\alpha) \in C^2(\BR^\dd)$ for every $\alpha \in \II$. 
It is now easily checked, using the properties of the matrix $Q(x)$ in Assumption \ref{A:SI}, that the \textit{jump} part $Qf(x,\alpha)$ of the generator $\gen f(x,\alpha)$ satisfies
\begin{equation}\label{E:jump-gen-micro}
Qf(x,\alpha) \triangleq \sum_{\substack{\beta \in \II\\\beta \neq \alpha}} q_{\alpha\beta}(x) \left[f(x,\beta)-f(x,\alpha)\right] = \sum_{\beta =1}^\mm q_{\alpha\beta}(x) \left[f(x,\beta)-f(x,\alpha)\right] = \sum_{\beta =1}^\mm q_{\alpha\beta}(x) f(x,\beta).
\end{equation}
Similar considerations apply to the operators $\gen^*$ and $\gen^\eps$ defined in \eqref{E:adjoint} and \eqref{E:macro-gen}, respectively.
\end{remark}

Pivotal to stating and then proving our main result is Proposition \ref{P:Fredholm} below, which states a Fredholm alternative for solvability of {\sc pde} of the form $\gen u(x,\alpha)=\phi(x,\alpha)$, $\gen^* m(x,\alpha)=\psi(x,\alpha)$, with periodic boundary conditions (in $x$). Owing to this periodicity, we will frequently find it convenient to think of the spatial variable $x$ as taking values in the $\dd$-dimensional torus $\BT^\dd \triangleq \BR^\dd/\BZ^\dd$. The latter consists of all equivalence classes under the relation $\sim$ on $\BR^\dd$  defined by $x \sim y$, $x,y \in \BR^\dd$, if and only if $x-y \in \BZ^\dd$. We define the projection map $\pi:\BR^\dd \to \BT^\dd$ that maps each point in $\BR^\dd$ to the corresponding equivalence class under $\sim$, i.e., $\pi(x)=[x]=x+\BZ^\dd \triangleq \{x+k:k \in \BZ^\dd\}$. 

Armed with the notation above, we can now think of a $1$-periodic function $g:\BR^\dd \to \BR$ as the function $\bar{g}:\BT^\dd \to \BR$ defined by $\bar{g}(\pi(x))=g(x)$. The same recipe also allows us to extend a function $\bar{h}:\BT^\dd \to \BR$ to yield a $1$-periodic function $h:\BR^\dd \to \BR$. We will say that a function $\bar g \in C^k(\BT^\dd)$, $k \in \BZ^+$, if $\bar g \circ \pi$ is a $1$-periodic $C^k$ function on $\BR^\dd$. More generally, given a function $f:\BR^\dd \times \II \to \BR$ which is $1$-periodic in $x$ for each $\alpha \in \II$, we define $\bar{f}:\BT^\dd \times \II \to \BR$ by setting $\bar{f}(\pi(x),\alpha)=f(x,\alpha)$. Further, we will say that the function $\bar{f} \in C^k(\BT^\dd \times \II)$, $k \in \BZ^+$, if for each $\alpha \in \II$, the function $\bar f \circ \pi$ is a $1$-periodic $C^k$ function in $x$ on $\BR^\dd$. As usual, the term smooth will imply membership in $C^k$ for every $k \in \BZ^+$. For ease of notation, we will frequently go back and forth between viewing periodic functions as having the spatial variable taking values in the whole space $\BR^\dd$ versus the torus $\BT^\dd$, and will simply write $f$, $g$, etc. in either case.

\begin{proposition}[Fredholm alternative]\label{P:Fredholm}
Suppose that Assumptions \ref{A:DD}, \ref{A:SI}, \ref{A:Periodicity} hold. Then, there exists a unique function $m:\BR^\dd \times \II \to [0,\infty)$ such that $x \mapsto  m(x,\alpha) \in C^2$ for every $\alpha \in \II$ and 
\begin{equation}\label{E:inv-dens-pde}
(\gen^* m)(x,\alpha)=0 \quad \text{for all $(x,\alpha) \in \BR^\dd\times\II$,} \quad \text{$m$ is $1$-periodic in $x$}, \quad \sum_{\alpha=1}^\mm \int_{\BT^\dd} m(x,\alpha) \thinspace dx=1.
\end{equation}
Let $\varphi:\BR^\dd \times \II \to \BR$ be a smooth function which is $1$-periodic in $x$. Then, there exists a unique function $\Phi(x,\alpha):\BR^\dd \times \II \to \BR$ with $x \mapsto \Phi(x,\alpha) \in C^2$ for each $\alpha \in \II$ solving the {\sc pde}
\begin{equation}\label{E:Fredholm-cell}
(\gen \Phi)(x,\alpha)=\varphi(x,\alpha) \quad \text{for all $(x,\alpha) \in \BR^\dd\times\II$,} \quad \text{$\Phi$ is $1$-periodic in $x$}, \quad \sum_{\alpha=1}^\mm \int_{\BT^\dd} \Phi(x,\alpha) \thinspace dx=0
\end{equation}
if and only if
\begin{equation}\label{E:centering}
\sum_{\alpha=1}^\mm \int_{\BT^\dd} \varphi(x,\alpha) m(x,\alpha) \thinspace dx=0.
\end{equation}
\end{proposition}

\begin{remark}\label{R:Fredholm}
While the proof of Proposition \ref{P:Fredholm} will be provided later in Section \ref{S:Fredholm}, we pause to make some important observations and comment on the significance of Assumptions \ref{A:DD}--\ref{A:Periodicity}.  
We start by noting that the function $m(x,\alpha)$ in \eqref{E:inv-dens-pde} admits an interpretation as the invariant density of the process $(\pi(X^\eps_t),I^\eps_t)$ taking values in $\BT^\dd \times \II$, which is obtained by projecting $X^\eps_t$ on to the torus. 
Turning to the issue of Assumptions \ref{A:DD}--\ref{A:Periodicity}, the uniform ellipticity of the diffusion matrices enables us to easily study solvability questions for the {\sc pde} in \eqref{E:inv-dens-pde} and \eqref{E:Fredholm-cell} using variational techniques, mimicking the arguments in \cite[Chapter 6]{Evans-PDE}. Together with the ellipticity, the irreducibility of the matrix $Q(x)$ of jump intensities and the smoothness of all coefficients allow us to use results from \cite{AllaireHutridurga-DCDSB-2015}, \cite{PeletierSchlottke_EJP2024}  (which in turn build on work of \cite{Sweers}) to assert that smooth solutions to \eqref{E:inv-dens-pde} exist and are unique up to normalization. Finally, the smoothness of coefficients ensures that the solutions to the {\sc pde} \eqref{E:Fredholm-cell}, when they exist, are at least $C^2$, thus rendering them suitable for use in the Ito formula.
\end{remark}

A straightforward application of Proposition \ref{P:Fredholm} now provides us a solution to the \textit{cell problem} codified in \eqref{E:cell-problem-i} below.

\begin{corollary}[Solution to the cell problem]\label{C:cell}
Let $\overline{b}=(b_1,\dots,b_\dd)$ be the average of the drift vector $b(x,\alpha)$ with respect to the invariant density $m(x,\alpha)$ given by \eqref{E:inv-dens-pde}, i.e.,
\begin{equation}\label{E:b-bar}
\overline{b} \triangleq \sum_{\alpha=1}^\mm \int_{\BT^\dd} b(x,\alpha) m(x,\alpha) \thinspace dx.
\end{equation}
Then, for $1 \le k \le \dd$, there exists a unique $\Phi_k:\BR^\dd \times \II \to \BR$ with $x \mapsto \Phi_k(x,\alpha) \in C^2(\BT^\dd)$ for every $\alpha \in \II$ solving the cell problem
\begin{equation}\label{E:cell-problem-i}
(\gen \Phi_k)(x,\alpha)=b_k(x,\alpha)-\overline{b}_k \quad \text{for all $(x,\alpha) \in \BR^\dd\times\II$,} \quad \text{$\Phi_k$ is $1$-periodic}, \quad \sum_{\alpha=1}^\mm \int_{\BT^\dd} \Phi_k(x,\alpha) \thinspace dx=0.
\end{equation}
\end{corollary}

To precisely formulate our main results, we need a couple more pieces of notation. First, we recall from \eqref{E:jump-gen-micro} the jump part $Qf$ of the generator $\gen f$, and define, for functions $f,g:\BR^\dd \times \II \to \BR$, the \textit{carr\'e-du-champ} operator 
\begin{multline}\label{E:cdc}% operateur carre du champ
\scQ(f,g)(x,\alpha) \triangleq Q(fg)(x,\alpha)-f(x,\alpha)\thinspace Qg(x,\alpha) - g(x,\alpha) \thinspace Qf(x,\alpha)\\
=\sum_{\beta=1}^\mm q_{\alpha\beta}\left(x\right)\left\{f\left(x,\beta\right)-
f\left(x,\alpha\right)\right\} \left\{g\left(x,\beta\right)-
g\left(x,\alpha\right)\right\}.
\end{multline}
Next, we introduce some spaces of paths corresponding to the trajectories of our stochastic processes of interest. Let $\scD([0,\infty);\BR^\dd)$ be the space of functions mapping $[0,\infty)$ to $\BR^\dd$ which are right-continuous with left limits, equipped with the Skorohod metric \cite[Section 3.5]{EK86}. Let $\scC([0,\infty);\BR^\dd)$ be the closed subset of $\scD([0,\infty);\BR^\dd)$ consisting of all continuous functions, noting \cite[Problem 3.11.25]{EK86} that the Skorohod metric, restricted to $\scC([0,\infty);\BR^\dd)$, generates the topology of uniform convergence on compact intervals of $[0,\infty)$. Finally, we will use the symbol $\Longrightarrow$ to denote convergence in distribution for processes with sample paths in the above function spaces.

\begin{theorem}[Main Result]\label{T:Main}
Let $\Phi:\BR^\dd \times \II \to \BR^\dd$ be defined by $\Phi(x,\alpha)\triangleq (\Phi_1(x,\alpha),\dots,\Phi_\dd(x,\alpha))$, where for $1 \le k \le \dd$, $\Phi_k(x,\alpha)$ solves the cell problem \eqref{E:cell-problem-i}. Let $\ttC=\left[\ttC_{kl}\right]_{1 \le k,l \le \dd}$ be the symmetric nonnegative-definite $\dd \times \dd$ matrix defined by
\begin{multline}\label{E:limcov}
\ttC \triangleq \sum_{\alpha=1}^\mm \int_{\BT^\dd}\biggl[ \Bigl((I_\dd - D\Phi)\thinspace a \thinspace (I-D\Phi)^\intercal\Big)(x,\alpha) + \scQ(\Phi,\Phi)(x,\alpha) \biggr] m(x,\alpha) \thinspace dx \quad \text{where}\\
\scQ(\Phi,\Phi)(x,\alpha) \triangleq \left[\scQ(\Phi_k,\Phi_l)(x,\alpha)\right]_{1 \le k,l \le \dd},
\end{multline}
%\begin{equation}\label{E:limcov}
%\ttC_{kl} \triangleq \sum_{\alpha=1}^\mm \int_{\BT^\dd}\biggl[ \Bigl((I_\dd - D\Phi)\thinspace a \thinspace (I-D\Phi)^\intercal\Big)_{kl}(x,\alpha) + \scQ(\Phi_k,\Phi_l)(x,\alpha) \biggr] m(x,\alpha) \thinspace dx
%\end{equation}
$D\Phi(x,\alpha)$ is the Jacobian (in $x$) of $\Phi$ with $k$-th row given by $\nabla \Phi_k(x,\alpha)$, $1 \le k \le \dd$, and $I_\dd$ denotes the $\dd \times \dd$ identity  matrix. Let $\overline{b} \triangleq \sum_{\alpha=1}^\mm \int_{\BT^\dd} b(x,\alpha) m(x,\alpha) \thinspace dx$,  as in Corollary \ref{C:cell}. Then,
\begin{equation*}
X^\eps_t - \frac{\overline{b}\thinspace t}{\eps} \Longrightarrow \sqrt{\ttC}B_t \qquad \text{in the space $\scC([0,\infty);\BR^\dd)$,}
\end{equation*}
where $B_t$ is a $\dd$-dimensional Brownian motion.
\end{theorem}

\begin{remark}\label{R:Interpretation}
If we compare the expression for the limiting diffusion matrix $\ttC$ defined by \eqref{E:limcov} in Theorem \ref{T:Main} above with the classical counterpart for a single {\sc sde} (see equation (4.3.10) in \cite{BLP}), a couple of things become apparent. First, each component $\Phi_k$ of the solution to the cell problem in our study now solves a system of, rather than a single, {\sc pde}, and, not surprisingly, one also has a sum over $\alpha$ when averaging with respect to the invariant density $m(x,\alpha)$. Second, and more interesting, is the appearance of the new term $\scQ(\Phi,\Phi)$ in the integrand on the right-hand side of the first line in \eqref{E:limcov}. The matrix $\sum_{\alpha=1}^\mm \int_{\BT^\dd} \scQ(\Phi,\Phi)(x,\alpha) \thinspace m(x,\alpha) \thinspace dx$ gives the extra contribution of the switching to the limiting diffusion  matrix $\ttC$, alluded to in Section \ref{S:Introduction}.
\end{remark}

\begin{remark}\label{R:Comparison}
We comment briefly on the relation of the present work with that in \cite{PeletierSchlottke_EJP2024}, whose setting has significant overlap with, but is not identical to, the one considered here. Indeed, while the continuous component in \cite{PeletierSchlottke_EJP2024} is allowed to take values in somewhat more general spaces, the discussion there for the Euclidean setting seems to be restricted to the case of constant diffusion coefficients. In contrast, we entirely restrict attention to the Euclidean setting, but on the other hand, allow fairly general \textit{variable} diffusion coefficients. In terms of results, the calculations in \cite{PeletierSchlottke_EJP2024} use an Eulerian scaling and give the exponential asymptotics of the probabilities of \textit{rare} events, i.e., $\BP\left(\eps X^\eps_{\cdot} \in B\right)$ with $B$ a subset of path space, asserting along the way the convergence of $\eps X^\eps_t$ to $\overline{b}\thinspace t$ in the limit as $\eps \searrow 0$, the latter admitting an interpretation as a law of large numbers ({\sc lln}). The present work uses a diffusive scaling with our main result Theorem \ref{T:Main} asserting, very roughly, that $\eps X^\eps_t \approx \overline{b}\thinspace t + \eps \sqrt{\ttC}B_t$, thus quantifying behavior of \textit{typical} fluctuations in the spirit of the {\sc clt}.
\end{remark}

Our thoughts are organized as follows. The principal tool for proving Theorem \ref{T:Main} is the Fredholm alternative stated in Proposition \ref{P:Fredholm} and its immediate Corollary \ref{C:cell}. Since Proposition \ref{P:Fredholm} is a purely \textit{analytic} result (at least overtly), we defer its proof to Section \ref{S:Fredholm}, focusing instead in Section \ref{S:Homogenization} on proving our main \textit{probabilistic} result Theorem \ref{T:Main} using these analytic tools.

% Proof of main result (Homogenization)
\section{Homogenization}\label{S:Homogenization}
In this section, we prove Theorem \ref{T:Main} with the help of Proposition \ref{P:Fredholm} (and Corollary \ref{C:cell}) whose proof is provided in Section \ref{S:Fredholm}. Owing to its centrality in our calculations, we start by stating the Ito formula \eqref{E:Ito-macro} which describes the action of smooth functions $f(x,\alpha)$ on the process $(X^\eps_t,I^\eps_t)$ solving \eqref{E:sde-prm-macro}, in terms of the generator \eqref{E:macro-gen} and associated martingales. After expressing our process of interest $X^\eps_t - \overline{b} \thinspace t/\eps$ as a small perturbation of a vector martingale $M^\eps_t$ in Lemma \ref{L:mart-decomp}, we compute in Lemma \ref{L:CV} the cross-variation processes associated with the components of $M^\eps_t$. This sets the stage for use of the martingale {\sc clt}, (an adaptation of) which is stated for reference in Proposition \ref{P:martclt}. Carrying out this program requires the ergodicity result Lemma \ref{L:ergodicity} which enables us to approximate, in an $L^2$ sense, time integrals of functions of $(\frac{X^\eps_t}{\eps},I^\eps_t)$ with spatial averages of these functions with respect to the invariant density $m(x,\alpha)$ in Proposition \ref{P:Fredholm}. With all the pieces in place, we close out the section by providing the proof of Theorem \ref{T:Main}. 

The generator $\gen^\eps$ of the process $(X^\eps_t,I^\eps_t)$ is an operator whose action on functions $f:\BR^\dd \times \II \to \BR$ which are $C^2$ in $x$ for every $\alpha \in \II$, is given by the function $\gen^\eps f:\BR^\dd \times \II \to \BR$ defined according to
\begin{multline}\label{E:macro-gen}
\gen^\eps f(x,\alpha) \triangleq \frac{1}{\eps}\sum_{i=1}^\dd b_i\left(\frac{x}{\eps},\alpha\right)\frac{\partial f}{\partial x_i}(x,\alpha) + \frac{1}{2}\sum_{i,j=1}^\dd a_{ij}\left(\frac{x}{\eps},\alpha\right) \frac{\partial^2 f}{\partial x_i \partial x_j}(x,\alpha)\\ + \frac{1}{\eps^2}\sum_{\beta=1}^\mm q_{\alpha\beta}\left(\frac{x}{\eps}\right) \left[f(x,\beta)-f(x,\alpha)\right].
\end{multline}
One can now see from the Ito formula \cite[Theorem II.5.1]{IkedaWatanabe} that for $f$ as above, we have
\begin{equation}\label{E:Ito-macro}
\begin{aligned}
f(X^\eps_t,I^\eps_t) &= f(X^\eps_0,I^\eps_0) + \int_0^t (\gen^\eps f)(X^\eps_s,I^\eps_s) \thinspace ds + \ttM^{1,\eps,f}_t + \ttM^{2,\eps,f}_t, \qquad \text{where}\\
\ttM^{1,\eps,f}_t &\triangleq \sum_{i=1}^\dd \sum_{j=1}^\rr \int_0^t \frac{\partial f}{\partial x_i}(X^\eps_s,I^\eps_s) \sigma_{ij}\left(\frac{X^\eps_s}{\eps},I^\eps_s\right) \thinspace dW^{(j)}_s, \qquad \text{and}\\
\ttM^{2,\eps,f}_t &\triangleq \int_0^{t+} \int_{\BR}\left\{f\left(X^\eps_{s-},I^\eps_{s-}+h^\eps(X^\eps_{s-},I^\eps_{s-},z)\right) - f\left(X^\eps_{s-},I^\eps_{s-}\right) \right\} \widetilde{\NN}_p(ds,dz),
\end{aligned}
\end{equation} 
where $\widetilde{\NN}_p(ds,dz)$ is the compensated {\sc prm} defined in \eqref{E:cprm}.
We will make extensive use of \eqref{E:Ito-macro} with the function $f(x,\alpha)$ frequently taken to be the smooth solution of a (judiciously chosen) system of {\sc pde}. Indeed, our very first step is to use the solution to the cell problem in Corollary \ref{C:cell} with the Ito formula to obtain a martingale decomposition of the process $X^\eps_t - \overline{b} \thinspace t/\eps$ of interest.

\begin{lemma}\label{L:mart-decomp}
Let $\Phi(x,\alpha)\triangleq (\Phi_1(x,\alpha),\dots,\Phi_\dd(x,\alpha))$ where for $1 \le k \le \dd$, $\Phi_k(x,\alpha)$ solves the cell problem \eqref{E:cell-problem-i}. Set $\Phi^\eps_k(x,\alpha) \triangleq \eps \thinspace\Phi_k\left(\frac{x}{\eps},\alpha\right)$ for $1 \le k \le \dd$. Then, we have
\begin{equation}\label{E:mart-decomp-1}% Decomposing process of interest in terms of martingales
X^{\eps}_t - \frac{\overline{b}\thinspace t}{\eps} = X^{\eps}_0 + \eps \left\{\Phi \left(\frac{X^\eps_t}{\eps},I^\eps_t\right) - \Phi \left(\frac{X^\eps_0}{\eps},I^\eps_0\right)\right\} + M^\eps_t, 
\end{equation}
where the process $M^\eps_t \triangleq (M^{\eps,1}_t,\dots,M^{\eps,\dd}_t)$ is a square integrable martingale with sample paths in $\scD([0,\infty);\BR^\dd)$ and components
\begin{equation}\label{E:cons-mart} % consolidated martingale
%\begin{aligned} 
M^{\eps,k}_t \triangleq \widetilde{\ttM}^{1,\eps,\Phi^\eps_k}_t - \ttM^{2,\eps,\Phi^\eps_k}_t, \quad
\text{where} \quad 
\widetilde{\ttM}^{1,\eps,\Phi^\eps_k}_t \triangleq \sum_{j=1}^\rr \int_0^t \left[(I_\dd-D\Phi)\thinspace \sigma\right]_{kj}\left(\frac{X^\eps_s}{\eps},I^\eps_s\right) dW^{(j)}_s.
%\end{aligned}
\end{equation}
\end{lemma}

\begin{proof}
We start by setting $\Phi^\eps_k(x,\alpha) \triangleq \eps \thinspace\Phi_k\left(\frac{x}{\eps},\alpha\right)$ and note that $(\gen^\eps \Phi^\eps_k)(x,\alpha) = \frac{1}{\eps} (\gen \Phi_k)\left(\frac{x}{\eps},\alpha\right)$. By Ito's formula \eqref{E:Ito-macro}, we get 
%\begin{equation*}
$\Phi^\eps_k(X^\eps_t,I^\eps_t) = \Phi^\eps_k(X^\eps_0,I^\eps_0) + \frac{1}{\eps}\int_0^t (\gen \Phi_k)\left(\frac{X^\eps_s}{\eps},I^\eps_s\right) \thinspace ds + \ttM^{1,\eps,\Phi^\eps_k}_t + \ttM^{2,\eps,\Phi^\eps_k}_t$ 
%\end{equation*}
where $\ttM^{1,\eps,\Phi^\eps_k}_t$, $\ttM^{2,\eps,\Phi^\eps_k}_t$ are computed as in \eqref{E:Ito-macro}. Using \eqref{E:cell-problem-i} and the {\sc sde} \eqref{E:sde-prm-macro}, we easily get the stated result.
\end{proof}

We next compute the cross-variation processes associated with the martingales $M^{\eps,k}_t$, $1 \le k \le \dd$. 

\begin{lemma}\label{L:CV} % CV = cross-variation processes
Let $\Phi(x,\alpha)\triangleq (\Phi_1(x,\alpha),\dots,\Phi_\dd(x,\alpha))$ be as in Lemma \ref{L:mart-decomp}. 
For $1 \le k,l \le \dd$, we have
\begin{equation}\label{E:CV}
\la M^{\eps,k},M^{\eps,l} \ra_t = \int_0^t \left((I_\dd-D\Phi) \thinspace a \thinspace (I_\dd-D\Phi)^\intercal\right)_{kl}\left(\frac{X^\eps_s}{\eps},I^\eps_s\right) \thinspace ds + \int_0^t \scQ(\Phi_k,\Phi_l)\left(\frac{X^\eps_s}{\eps},I^\eps_s\right) \thinspace ds,
\end{equation}
with $\scQ(\Phi_k,\Phi_l)$ computed as in \eqref{E:cdc}.
\end{lemma} 

\begin{proof}
For $1 \le k \le \dd$, $(x,\alpha) \in \BR^\dd \times \II$, $z \in \BR$, set 
\begin{equation}\label{E:varDelta}
\varDelta \Phi^\eps_k(x,\alpha,z) \triangleq \Phi^\eps_k\left(x,\alpha+h^\eps(x,\alpha,z)\right) - \Phi^\eps_k\left(x,\alpha\right) \quad \text{where  $\Phi^\eps_k(x,\alpha) \triangleq \eps \thinspace\Phi_k\left(\frac{x}{\eps},\alpha\right)$.}
\end{equation}
We first compute $\la M^{\eps,k} \ra_t$. Note that $\left(M^{\eps,k}_t\right)^2 = \left(\widetilde{\ttM}^{1,\eps,\Phi^\eps_k}_t\right)^2 + \left(\ttM^{2,\eps,\Phi^\eps_k}_t\right)^2 - 2\thinspace \widetilde{\ttM}^{1,\eps,\Phi^\eps_k}_t \cdot \ttM^{2,\eps,\Phi^\eps_k}_t$. 
We claim that
\begin{multline}\label{E:M_12}
\la \widetilde{\ttM}^{1,\eps,\Phi^\eps_k} \ra_t = \int_0^t \left[(I-D\Phi)\thinspace a \thinspace (I-D\Phi)^\intercal\right]_{kk} \left(\frac{X^\eps_s}{\eps},I^\eps_s\right) \thinspace ds \quad \text{and}\\ 
\la \ttM^{2,\eps,\Phi^\eps_k} \ra_t = \int_0^t \int_{\BR} \left(\varDelta \Phi^\eps_k(X^\eps_{s-},I^\eps_{s-},z)\right)^2  dz \thinspace ds.
\end{multline}
The first assertion in \eqref{E:M_12} follows from standard results on stochastic integrals with respect to Brownian motion, while the second follows by noting that $\ttM^{2,\eps,\Phi^\eps_k}_t = \int_0^{t+}\int_{\BR} \varDelta \Phi^\eps_k\left(X^\eps_{s-},I^\eps_{s-},z\right)\widetilde{\NN}_p(ds,dz)$ and using results from \cite[Section II.3]{IkedaWatanabe}. Next, an application of the Ito formula \cite[Theorem II.5.1]{IkedaWatanabe} to $f(\widetilde{\ttM}^{1,\eps,\Phi^\eps_k}_t,\ttM^{2,\eps,\Phi^\eps_k}_t)$ with $f(x_1,x_2) \triangleq x_1\thinspace x_2$ yields that
\begin{equation}\label{E:M12-product}
\widetilde{\ttM}^{1,\eps,\Phi^\eps_k}_t \cdot \ttM^{2,\eps,\Phi^\eps_k}_t = \int_0^t \ttM^{2,\eps,\Phi^\eps_k}_s \thinspace d\widetilde{\ttM}^{1,\eps,\Phi^\eps_k}_s + \int_0^{t+}\int_{\BR} \widetilde{\ttM}^{1,\eps,\Phi^\eps_k}_{s-} \varDelta \Phi^\eps_k(X^\eps_{s-},I^\eps_{s-}z) \thinspace \widetilde{\NN}_p(ds,dz) \quad \text{is a martingale.}
\end{equation}
It now follows from \eqref{E:M_12} and \eqref{E:M12-product} that 
\begin{equation*}%\label{E:qv-generic} % generic quadratic variation
\la M^{\eps,k} \ra_t = \int_0^t \left[(I-D\Phi)\thinspace a \thinspace (I-D\Phi)^\intercal\right]_{kk} \left(\frac{X^\eps_s}{\eps},I^\eps_s\right) \thinspace ds + \int_0^t \int_{\BR} \left(\varDelta \Phi^\eps_k(X^\eps_{s-},I^\eps_{s-},z)\right)^2  dz \thinspace ds.
\end{equation*}
Noting that for $1 \le k,l \le \dd$, we have 
%\begin{equation*}
$\la M^{\eps,k},M^{\eps,l} \ra_t \triangleq \frac{1}{4}\left(\la M^{\eps,k} + M^{\eps,l}\ra_t - \la M^{\eps,k} - M^{\eps,l}\ra_t\right)$, we get
%\end{equation*}
%\begin{multline*}
\begin{multline} \label{E:cv}
\la M^{\eps,k},M^{\eps,l} \ra_t = \int_0^t \left((I_\dd-D\Phi) \thinspace a \thinspace (I_\dd-D\Phi)^\intercal\right)_{kl}\left(\frac{X^\eps_s}{\eps},I^\eps_s\right) \thinspace ds\\ + \int_0^t \int_{\BR} \varDelta \Phi^\eps_k(X^\eps_{s-},I^\eps_{s-},z) \thinspace \varDelta \Phi^\eps_l(X^\eps_{s-},I^\eps_{s-},z) \thinspace dz \thinspace ds.
\end{multline}
%\end{multline*}
Note that the rightmost term in \eqref{E:cv} can be expressed as
%\begin{equation*}
$\int_0^t \Theta^\eps_{kl}(X^\eps_{s-},I^\eps_{s-})\thinspace ds$ where we have defined $\Theta^\eps_{kl}(x,\alpha) \triangleq \int_{\BR} \varDelta \Phi^\eps_k(x,\alpha,z) \varDelta \Phi^\eps_l(x,\alpha,z) \thinspace dz$.
%\end{equation*}
To evaluate this, we use \eqref{E:varDelta} in conjunction with the expression for $h^\eps(x,\alpha,z)$ in \eqref{E:h-eps} to get
%\begin{equation*}
%\begin{aligned}
$\Theta^\eps_{kl}(x,\alpha) = \sum_{\beta=1}^\mm \int_{\ttJ^\eps_{\alpha\beta}(x)} \left\{\Phi^\eps_k(x,\beta)-\Phi^\eps_k(x,\alpha)\right\}\left\{\Phi^\eps_l(x,\beta)-\Phi^\eps_l(x,\alpha)\right\} \thinspace dz\\
= \sum_{\beta=1}^\mm q^\eps_{\alpha\beta}(x)\left\{\Phi^\eps_k(x,\beta)-\Phi^\eps_k(x,\alpha)\right\}\left\{\Phi^\eps_l(x,\beta)-\Phi^\eps_l(x,\alpha)\right\}$.
%\end{aligned}
%\end{equation*}
Recalling the expression for $q^\eps_{\alpha\beta}(x)$ from \eqref{E:sw-int-macro}, and using the second equation in \eqref{E:varDelta}, we get that $\Theta^\eps_{kl}(x,\alpha) = \scQ(\Phi_k,\Phi_l)(\frac{x}{\eps},\alpha)$. Thus, the second term on the right-hand side of \eqref{E:cv} can be written $\int_0^t \scQ(\Phi_k,\Phi_l)\left(\frac{X^\eps_{s-}}{\eps},I^\eps_{s-}\right) \thinspace ds$. Noting that $X^\eps_{s-}=X^\eps_s$, $I^\eps_{s-}=I^\eps_s$ for a.e. $s$, we obtain
\begin{equation}\label{E:qcv-M2}
\int_0^t \int_{\BR} \varDelta \Phi^\eps_k(X^\eps_{s-},I^\eps_{s-},z) \thinspace \varDelta \Phi^\eps_l(X^\eps_{s-},I^\eps_{s-},z) \thinspace dz \thinspace ds=\int_0^t \scQ(\Phi_k,\Phi_l)\left(\frac{X^\eps_s}{\eps},I^\eps_s\right) \thinspace ds.
\end{equation} 
Putting together \eqref{E:cv} and \eqref{E:qcv-M2}, we now get the stated result.
\end{proof}

%\newpage

A key role in the proof of Theorem \ref{T:Main} will be played by the \textit{martingale central limit theorem}. A version tailored to our problem, adapted from Theorems 7.1.1 and 7.1.4 in \cite{EK86}, is stated below.

\begin{proposition}[Martingale {\sc clt}]\label{P:martclt}
Let $(\Omega,\filt,\BP)$ be a probability space equipped with a family of filtrations $\{\filt^\eps_t:t \ge 0\}$, $\eps \in (0,1)$, supporting a family of processes $\{M^\eps_t:t \ge 0\}$, $\eps \in (0,1)$, where each $M^\eps_t=(M^{\eps,1}_t,\dots,M^{\eps,\dd}_t)$ is an $\BR^\dd$-valued square-integrable $\filt^\eps_t$-martingale\footnote{Thus, each component $M^{\eps,k}_t$ is a square-integrable $\filt^\eps_t$-martingale.} with sample paths in $\scD([0,\infty);\BR^\dd)$ satisfying $M^\eps_0=0$. 
Let $A^\eps=\{A^\eps_t=[A^{\eps,kl}_t]_{1 \le k,l \le \dd}:t \ge 0\}$ be symmetric $\dd \times \dd$ matrix-valued processes such that $A^{\eps,kl}_t$ has sample paths in $\scD([0,\infty);\BR)$ and $A^\eps_t-A^\eps_s$ is nonnegative-definite for $t>s \ge 0$. Suppose that for each $T>0$, $1 \le k,l \le \dd$, we have
\begin{equation}\label{E:mart-clt-tc} % martingale clt technical conditions
\lim_{\eps \searrow 0} \BE\left[\sup_{0 \le t \le T} \left|A^{\eps,kl}_t - A^{\eps,kl}_{t-}\right| \right]=0, \quad \text{and} \quad 
\lim_{\eps \searrow 0} \BE\left[\sup_{0 \le t \le T} \left|M^\eps_t - M^\eps_{t-}\right|^2 \right]=0,
\end{equation}
and
\begin{equation*}
M^{\eps,k}_t \cdot M^{\eps,l}_t - A^{\eps,kl}_t \quad \text{is an $\filt^\eps_t$-martingale.}
\end{equation*}
Let $C=[c_{kl}]_{1 \le k,l \le \dd}$ be a symmetric nonnegative-definite $\dd \times \dd$ matrix and suppose that for each $t \ge 0$, $1 \le k,l \le \dd$,
\begin{equation*}
A^{\eps,kl}_t \to c_{kl} \thinspace t \quad \text{in probability, as $\eps \searrow 0$.}
\end{equation*}
Then, $M^\eps \Longrightarrow \overline{X}$ in the space $\scD([0,\infty);\BR^\dd)$, where $\overline{X}$ is the unique (in distribution) process with independent Gaussian increments having sample paths in $\scC([0,\infty);\BR^\dd)$ such that $\overline{X}^{k}_t$ and $\overline{X}^{k}_t \cdot \overline{X}^{l}_t - c_{kl} \thinspace t$ are $\filt^{\overline{X}}_t$-martingales, i.e., $M^\eps \Longrightarrow \sqrt{C}B_t$, where $B_t$ is $\dd$-dimensional Brownian motion.
\end{proposition}

\begin{lemma}[Ergodicity]\label{L:ergodicity}
Let $g:\BT^\dd \times \II \to \BR$ be smooth. Then, for any $t \ge 0$, we have 
\begin{equation}\label{E:ergodicity}
\lim_{\eps \searrow 0} \int_0^t g\left(\frac{X^\eps_s}{\eps},I^\eps_s\right) \thinspace ds = \overline{g} \thinspace t \quad \text{in $L^2$,} \quad \text{where} \quad \overline{g} \triangleq \sum_{\alpha=1}^\mm \int_{\BT^\dd} g(x,\alpha) m(x,\alpha) \thinspace dx,
\end{equation}
with $m(x,\alpha)$ as in Proposition \ref{P:Fredholm}.
\end{lemma}

\begin{proof}
Since the function $g(x,\alpha)-\overline{g}$ satisfies \eqref{E:centering}, it follows from Proposition \ref{P:Fredholm} that there exists a function $f(x,\alpha)$ with $x \mapsto f(x,\alpha) \in C^2$ for each $\alpha \in \II$, such that $\gen f(x,\alpha)=g(x,\alpha)-\overline{g}$ for all $(x,\alpha) \in \BR^\dd \times \II$, $f$ is $1$-periodic in $x$. Setting $f^\eps(x,\alpha) \triangleq \eps f\left(\frac{x}{\eps},\alpha\right)$, we note that $\gen^\eps f^\eps(x,\alpha)=\frac{1}{\eps} \gen f\left(\frac{x}{\eps},\alpha\right)$ and use Ito's formula \eqref{E:Ito-macro} to obtain
\begin{equation*}
\int_0^t \left\{g\left(\frac{X^\eps_s}{\eps},I^\eps_s\right) - \overline{g}\right\} \thinspace ds = \eps^2 \left\{f\left(\frac{X^\eps_t}{\eps},I^\eps_t\right) - f\left(\frac{X^\eps_0}{\eps},I^\eps_0\right)\right\} - \eps \thinspace \ttM^{1,\eps,f^\eps}_t - \eps \thinspace \ttM^{2,\eps,f^\eps}_t
\end{equation*}
where $\ttM^{i,\eps,f^\eps}_t$, $i \in \{1,2\}$, are as in \eqref{E:Ito-macro}. Now, $\la \ttM^{1,\eps,f^\eps} \ra_t = \sum_{j=1}^\rr \int_0^t \left\{\sum_{i=1}^\dd \frac{\partial f}{\partial x_i}\left(\frac{X^\eps_s}{\eps},I^\eps_s\right) \sigma_{ij}\left(\frac{X^\eps_s}{\eps},I^\eps_s\right)\right\}^2 \thinspace ds$. Next, it follows from \eqref{E:M_12} and \eqref{E:qcv-M2} that $\la \ttM^{2,\eps,f^\eps} \ra_t = \int_0^t \scQ(f,f)\left(\frac{X^\eps_s}{\eps},I^\eps_s\right) \thinspace ds$. Recalling the boundedness of $f$ and its partial derivatives, together with that of $\sigma_{ij}$, and noting that $\BE\left[\left(\ttM^{i,\eps,f^\eps}_t\right)^2\right] = \BE\left[ \la \ttM^{i,\eps,f^\eps} \ra_t \right]$ for $i \in \{1,2\}$, it now follows that there exists a constant $K>0$ such that
\begin{equation*}
\BE\left[\left|\int_0^t \left\{ g\left(\frac{X^\eps_s}{\eps},I^\eps_s\right) - \overline{g}\right\} \thinspace ds\right|^2 \right] \le K\eps^2 (t+\eps^2),
\end{equation*}
which easily yields the stated claim.
\end{proof}
%\textsf{\color{blue} (Write the ergodic lemma above in terms of the macroscopic processes $(X^\eps_t,I^\eps_t)$.) \color{black}}

We now prove our main homogenization result, viz., Theorem \ref{T:Main}.

\begin{proof}[Proof of Theorem \ref{T:Main}]
We begin by outlining in broad strokes the strategy and organization of the proof. The starting point is Lemma \ref{L:mart-decomp} which decomposes the process $X^\eps_t - \overline{b}\thinspace t/\eps$ as the sum of two terms: a small term $X^{\eps}_0 + \eps \left\{\Phi \left(\frac{X^\eps_t}{\eps},I^\eps_t\right) - \Phi \left(\frac{X^\eps_0}{\eps},I^\eps_0\right)\right\}$ and an $\BR^\dd$-valued martingale $M^\eps_t$. The aim is to show that the first of these terms converges to $0$ while the second converges to $\sqrt{\ttC} B_t$, and then argue that the sum, viz., $X^\eps_t - \overline{b}\thinspace t/\eps$ converges to $\sqrt{\ttC} B_t$. The bulk of the effort centers around proving that $M^\eps_t \Longrightarrow \sqrt{\ttC} B_t$ in the space $\scD([0,\infty);\BR^\dd)$ as $\eps \searrow 0$. This is accomplished via the Martingale {\sc clt} stated in Proposition \ref{P:martclt} by exploiting the expressions for the cross-varation processes $\la M^{\eps,k},M^{\eps,l}\ra_t$ obtained in Lemma \ref{L:CV}, together with the ergodicity results in Lemma \ref{L:ergodicity}.

We start by noting that the process $M^\eps_t \triangleq (M^{\eps,1}_t,\dots,M^{\eps,\dd}_t)$ from Lemma \ref{L:mart-decomp} with components defined in \eqref{E:cons-mart}, is a square-integrable $\filt_t$-martingale with sample paths in $\scD([0,\infty);\BR^\dd)$ satisfying $M^\eps_0=0$. 
Let $A^\eps = \{A^\eps_t=[A^{\eps,kl}_t]_{1 \le k,l \le \dd}:t \ge 0\}$ be the symmetric $\dd \times \dd$ matrix-valued process defined by $A^{\eps,kl}_t \triangleq \la M^{\eps,k},M^{\eps,l} \ra_t$ for $1 \le k,l \le \dd$, $t \ge 0$. Having set up the notation, the next several steps involve checking that the conditions for invoking the Martingale {\sc clt} are indeed satisfied.

Let $1 \le k,l \le \dd$. By Lemma \ref{L:CV}, the sample paths of $A^{\eps,kl}_t$ are continuous, and the process $M^{\eps,k}_t \cdot M^{\eps,l}_t - A^{\eps,kl}_t$ is clearly a martingale. We next turn to the matrix-valued process
\begin{equation}\label{E:A-matrix}
A^\eps_t = \int_0^t \left((I_\dd-D\Phi) \thinspace a \thinspace (I_\dd-D\Phi)^\intercal\right)\left(\frac{X^\eps_s}{\eps},I^\eps_s\right) \thinspace ds + \int_0^t \scQ(\Phi,\Phi)\left(\frac{X^\eps_s}{\eps},I^\eps_s\right) \thinspace ds,
\end{equation}
where $\scQ(\Phi,\Phi)$ is computed using \eqref{E:cdc} and the second equation in \eqref{E:limcov}. A simple computation shows that for any $(x,\alpha) \in \BR^\dd \times \II$ and any $\xi \in \BR^\dd$, we have
\begin{multline}\label{E:nonnegdef}
\xi^\intercal \left[(I_\dd - D\Phi)\thinspace a \thinspace (I-D\Phi)^\intercal\right](x,\alpha)\thinspace \xi + \xi^\intercal \scQ(\Phi,\Phi)(x,\alpha) \thinspace \xi \\= \left\{(I_\dd-D\Phi)^\intercal(x,\alpha)\thinspace \xi\right\}^\intercal a(x,\alpha) \left\{(I_\dd-D\Phi)^\intercal(x,\alpha) \thinspace\xi\right\} + \sum_{\beta=1}^\mm q_{\alpha\beta}(x)\left|\left(\Phi(x,\beta)-\Phi(x,\alpha)\right)^\intercal \xi\right|^2.
\end{multline}
Recalling the uniform ellipticity of $a(x,\alpha)$ in Assumption \ref{A:DD} and the nonnegativity of $q_{\alpha\beta}$ for $\beta \neq \alpha$ in Assumption \ref{A:SI}, it now easily follows from equations \eqref{E:A-matrix} and \eqref{E:nonnegdef} that for any $\xi \in \BR^\dd$, $t>s \ge 0$, we have $\xi^\intercal (A^\eps_t - A^\eps_s)\thinspace \xi \ge 0$, i.e., $A^\eps_t - A^\eps_s$ is nonnegative-definite for $t>s \ge 0$. 

Turning to the conditions in \eqref{E:mart-clt-tc}, we note that the first of these holds on account of continuity of $A^{\eps,kl}_t$. To check the second, observe that $|M^\eps_t - M^\eps_{t-}|^2 = \sum_{k=1}^\dd |M^{\eps,k}_t - M^{\eps,k}_{t-}|^2 \le 4\eps^2 \sum_{k=1}^\dd \sup_{(x,\alpha) \in \BR^\dd \times \II}\left|\Phi_k\left(\frac{x}{\eps},\alpha\right)\right|^2$, where the latter follows from equations \eqref{E:cons-mart} and (the last equation in) \eqref{E:Ito-macro} with $\Phi^\eps_k(x,\alpha)=\eps \Phi_k(x/\eps,\alpha)$. Since the functions $\Phi_k$, $1 \le k \le \dd$, are bounded (being smooth and $1$-periodic in $x$), we easily get that $\lim_{\eps \searrow 0} \BE\left[\sup_{0 \le t \le T} \left|M^\eps_t - M^\eps_{t-}\right|^2 \right]=0$ for any $T>0$, as required.

We now turn our attention to the matrix $\ttC=[\ttC_{kl}]_{1 \le k,l \le \dd}$ in \eqref{E:limcov}, which is easily seen to be symmetric. Multiplying both sides of \eqref{E:nonnegdef} by the positive function $m(x,\alpha)$, integrating over $x \in \BT^\dd$ and summing over $\alpha \in \II$, it follows once again from Assumptions \ref{A:DD} and \ref{A:SI} that $\xi^\intercal \ttC \thinspace \xi \ge 0$ for all $\xi \in \BR^\dd$, i.e., $\ttC$ is nonnegative-definite. 
Finally, we note that on account of Assumptions \ref{A:DD}--\ref{A:Periodicity}, and the findings of Corollary \ref{C:cell}, the components of the matrix-valued functions $(I_\dd - D\Phi)\thinspace a \thinspace (I-D\Phi)^\intercal(x,\alpha)$ and $\scQ(\Phi,\Phi)(x,\alpha)$ are smooth and $1$-periodic in $x$. This enables application of Lemma \ref{L:ergodicity} to the components of $A^\eps_t$ in \eqref{E:A-matrix}, allowing us to conclude that as $\eps \searrow 0$, $A^{\eps,kl}_t \to \ttC_{kl} t$ in $L^2$, and hence in probability, for every $t \ge 0$. 

Putting the pieces together, we now conclude from Proposition \ref{P:martclt} that $M^\eps_t \Longrightarrow \sqrt{\ttC} B_t$ in the space $\scD([0,\infty);\BR^\dd)$ as $\eps \searrow 0$. 

We would now like to argue that the process $X^\eps_t- \overline{b} \thinspace t/\eps$ which is, by Lemma \ref{L:mart-decomp}, the sum of $M^\eps_t$ and $X^\eps_0 + \eps \left\{\Phi \left(\frac{X^\eps_t}{\eps},I^\eps_t\right) - \Phi \left(\frac{X^\eps_0}{\eps},I^\eps_0\right)\right\}$, also converges to  
$\sqrt{\ttC} B_t$ in the space $\scD([0,\infty);\BR^\dd)$. To see this, let $d_\Sk$ denote the Skorohod metric \cite{EK86} on $\scD([0,\infty);\BR^\dd)$ and note that for $x,y \in \scD([0,\infty);\BR^\dd)$, we have $d_\Sk(x,y) \le \rho(x,y) \triangleq \int_0^\infty e^{-u} \left(\sup_{0 \le t \le u}|x(t)-y(t)|\right)\wedge 1 \thinspace du$. Owing to the boundedness of $\Phi$, we now have that $\rho\left(X^\eps_t- \overline{b} \thinspace t/\eps,M^\eps_t\right) \to 0$ in probability. By \cite[Corollary 3.3.3]{EK86}, it now follows that $X^\eps_t- \overline{b} \thinspace t/\eps \Longrightarrow \sqrt{\ttC}B_t$ in the space $\scD([0,\infty);\BR^\dd)$ as $\eps \searrow 0$. Since the probability laws of both $X^\eps_t - \overline{b}\thinspace t/\eps$ and $\sqrt{\ttC}B_t$ are supported on $\scC([0,\infty);\BR^\dd)$, we now use \cite[Problem 3.11.25]{EK86} to conclude that in fact $X^\eps_t - \overline{b}\thinspace t/\eps \Longrightarrow \sqrt{\ttC}B_t$ in the space $\scC([0,\infty);\BR^\dd)$ as $\eps \searrow 0$.
\end{proof}

% Section: Fredholm alternative
\section{Fredholm Alternative}\label{S:Fredholm}

Our main goal in this section is to prove Proposition \ref{P:Fredholm}. Thus, we would like to study solvability and regularity of solutions for  {\sc pde} of the form $\gen u(x,\alpha)=\phi(x,\alpha)$ or, equivalently,  
\begin{equation}\label{E:generic-elliptic-pde}
-\gen u(x,\alpha)=\psi(x,\alpha),
\end{equation}
the notational sign change being simply for convenience in applying {\sc pde} results. While the analysis proceeds along a well-worn route---naturally adapting the single elliptic {\sc pde} calculations in \cite[Chapter 6]{Evans-PDE}, \cite[Chapter 7]{PavliotisStuart} to the case of systems---we provide a reasonably detailed discussion owing to the importance of Proposition \ref{P:Fredholm} in furnishing (via Corollary \ref{C:cell}) an abundance of smooth functions for use in the Ito formula. 

To set the stage, we start by noting that any function $v(x,\alpha):\BR^\dd \times \II \to \BR$ can be thought of as a vector-valued function $\bm v(x) \triangleq (v_1(x),\dots,v_\mm(x)): \BR^\dd \to \BR^\mm$ where $v_\alpha(x) \triangleq v(x,\alpha)$. We will also find it convenient to write $b^\alpha(x) \triangleq b(x,\alpha)$, $a^\alpha(x) \triangleq a(x,\alpha)$ and occasionally think of $b$ and $a$ as families of functions on $\BR^\dd$ parametrized by $\alpha \in \II$. With this notation in place, we set 
\begin{equation}\label{E:PDE-system-operator}
L \triangleq \mathtt{diag}(L_1,\dots,L_\mm) \quad \text{where} \quad L_\alpha \triangleq \sum_{j=1}^\dd b_j^\alpha(x)\frac{\partial}{\partial x_j} + \frac{1}{2}\sum_{j,k=1}^\dd a^\alpha_{jk}(x) \frac{\partial^2}{\partial x_j \partial x_k} \quad \text{for $1 \le \alpha \le \mm$.}
\end{equation}
Recalling the matrix $Q(x) \triangleq [q_{\alpha\beta}(x)]_{1 \le \alpha,\beta \le \mm}$ from Assumption \ref{A:SI}, together with the fact that $q_{\alpha\alpha}(x)=-\sum_{\substack{\beta \in \II\\\beta \neq \alpha}}q_{\alpha\beta}(x)$ for $1 \le \alpha \le \mm$, it is now evident that the {\sc pde} \eqref{E:generic-elliptic-pde} is equivalent to the weakly coupled elliptic system
\begin{equation}\label{E:wces} %weakly coupled elliptic system
(-L - Q)\bm u(x) = \bm \psi(x) \quad \text{for $x \in \BR^\dd$,}
\end{equation}
where $\bm u$, $\bm \psi$ are in accordance with the aforementioned notation. Recalling \eqref{E:adjoint}, we note that the adjoint {\sc pde} $-\gen^* v(x,\alpha)=\zeta(x,\alpha)$ can be written as the system
\begin{equation}\label{E:wces-adjoint}
\begin{aligned}
(-L^* - Q^*)\bm v(x) &= \bm \zeta(x), \quad \text{where} \quad Q^* \triangleq Q^\intercal, \quad \text{and} \quad L^* \triangleq \mathtt{diag}(L_1^*,\dots,L_\mm^*)\\
\text{with} \quad L_\alpha^*  &\triangleq -\sum_{j=1}^\dd \frac{\partial}{\partial x_j} \left(b_j^\alpha(x) \medspace\cdot\right) + \frac{1}{2}\sum_{j,k=1}^\dd  \frac{\partial^2}{\partial x_j \partial x_k}\left(a^\alpha_{jk}(x) \medspace \cdot\right) \quad \text{for $1 \le \alpha \le \mm$.}
\end{aligned}
\end{equation}

Next, we introduce some Sobolev spaces of periodic functions, following the development in \cite{Temam}. Below, $\BT^\dd$ denotes the $\dd$-dimensional torus. 
For $k \in \BZ^+$, let $H^k_\per(\BT^\dd)$ denote the space of functions $u \in H^k_\loc(\BR^\dd)$ which are $1$-periodic. For a multiindex $\mi=(\mi_1,\dots,\mi_\dd) \in (\BZ^+)^\dd$, define $|\mi| \triangleq \sum_{j=1}^\dd \mi_j$. Define an inner product on $H^k_\per(\BT^\dd)$ by setting
%\begin{equation*}
$(u,v)_{H^k_\per(\BT^\dd)} \triangleq \sum_{|\mi| \le k} \int_{\BT^\dd} D^\mi u(x) \thinspace D^\mi v(x) \thinspace dx$ for $u,v \in H^k_\per(\BT^\dd)$
%\end{equation*}
and let $\|u\|_{H^k_\per(\BT^\dd)} \triangleq \sqrt{(u,u)_{H^k_\per(\BT^\dd)}}$ be the induced norm. Now, let $\HH^k$ denote the $\mm$-fold product of $H^k_\per(\BT^\dd)$, viz., $(H^k_\per(\BT^\dd))^\mm$, with the  inner product between $\bm u=(u_1,\dots,u_\mm)$ and $\bm v=(v_1,\dots,v_\mm) \in \HH^k$ defined by
%\begin{equation*}
$(\bm u,\bm v)_{\HH^k} \triangleq \sum_{\alpha=1}^\mm (u_\alpha,v_\alpha)_{H^k_\per(\BT^\dd)}$,
%\end{equation*}
and note that
%\begin{equation*}
$\|\bm u\|_{\HH^k}^2 = \sum_{\alpha=1}^\mm \|u_\alpha\|^2_{H^k_\per(\BT^\dd)}$.
%\end{equation*}
We will be most interested in the cases $\LL \triangleq \HH^0 = (L^2_\per(\BT^\dd))^\mm$, and $\HH^1$. In these cases, if $\bm u=(u_1,\dots,u_\mm)$, $\bm v=(v_1,\dots,v_\mm)$, then the respective inner products take the form  
\begin{equation*}
(\bm u,\bm v)_\LL = \sum_{\alpha=1}^\mm \int_{\BT^\dd} u_\alpha(x) v_\alpha(x) \thinspace dx , \qquad (\bm u,\bm v)_{\HH^1} = \sum_{\alpha=1}^\mm \int_{\BT^\dd} u_\alpha v_\alpha \thinspace dx + \sum_{\alpha=1}^\mm \int_{\BT^\dd} (\nabla u_\alpha(x), \nabla v_\alpha(x))_{\BR^\dd} \thinspace dx.
\end{equation*}
Note that for any $\bm u \in \HH^1$, we have $\|\bm u\|_\LL \le \|\bm u\|_{\HH^1}$.

To study solvability of \eqref{E:wces} within the weak (variational) formulation, we use the smoothness of the coefficients (Assumption \ref{A:DD}) to first write each $L_\alpha$ from \eqref{E:PDE-system-operator} in divergence form
\begin{equation*}
L_\alpha f(x) = \frac{1}{2} \sfdiv\left(a^\alpha(x) \nabla f(x)\right) + \left(\tilde{b}^\alpha(x),\nabla f(x)\right)_{\BR^\dd} \quad\text{where} \quad \tilde{b}^\alpha_j(x) \triangleq b^\alpha_j(x) - \frac{1}{2}\sum_{k=1}^\dd \frac{\partial}{\partial x_k}(a^\alpha_{kj}), \quad \text{$1 \le j \le \dd$,}
\end{equation*}
for $f:\BR^\dd \to \BR$ sufficiently regular, $1 \le \alpha \le \mm$, $x \in \BR^\dd$. We next introduce the bilinear form $\blf:\HH^1 \times \HH^1 \to \BR$ defined by
\begin{multline}\label{E:blf}
\blf[\bm u,\bm v] \triangleq \frac{1}{2}\sum_{\alpha=1}^\mm \int_{\BT^\dd} \left(a^\alpha(x) \nabla u_\alpha(x),\nabla v_\alpha(x)\right)_{\BR^\dd} dx - \sum_{\alpha=1}^\mm \int_{\BT^\dd} \left(\tilde{b}^\alpha(x),\nabla u_\alpha(x) \right)_{\BR^\dd} v_\alpha(x) \thinspace dx \\- \sum_{\alpha=1}^\mm \sum_{\beta=1}^\mm \int_{\BT^\dd} q_{\alpha\beta}(x) \left[u_\beta(x)-u_\alpha(x)\right] v_\alpha(x) \thinspace dx.
\end{multline}
Thus, $\blf[\bm u,\bm v]$ is obtained by multiplying the left-hand side of the $\alpha$-th equation in \eqref{E:wces} by $v_\alpha$, integrating over $\BT^\dd$, doing an integration by parts in the diffusive terms, and summing over $\alpha$, once again using $q_{\alpha\alpha}(x)=-\sum_{\substack{\beta \in \II\\\beta \neq \alpha}}q_{\alpha\beta}(x)$. 

\begin{definition}[Weak solution]\label{D:weak-solution}
Let $\bm \psi \in \LL$. By a weak solution of the system $(-L-Q)\bm u=\bm \psi$, we mean a function $\bm u \in \HH^1$ such that $\blf[\bm u,\bm v]=(\bm \psi,\bm v)_\LL$ for all $\bm v \in \HH^1$. More generally, if $\lambda \in \BR$ and $I_\mm$ denotes the $\mm \times \mm$ identity matrix, a weak solution to the system $(\lambda I_\mm - L - Q)\bm u = \bm \psi$ is a function $\bm u \in \HH^1$ satisfying $\blf_\lambda[\bm u,\bm v]=(\bm \psi,\bm v)_\LL$ for all $\bm v \in \HH^1$, where $\blf_\lambda[\bm u,\bm v] \triangleq \blf[\bm u,\bm v] + \lambda(\bm u,\bm v)_{\LL}$.
\end{definition}

\begin{remark}\label{R:adjoint-pde}
It is easily checked that the operators $L_\alpha^*$ defined in \eqref{E:wces-adjoint} can be expressed in divergence form as $L_\alpha^* v_\alpha(x) = \frac{1}{2}\mathsf{div}\left(a^\alpha(x)\nabla v_\alpha(x)\right) - \mathsf{div}\left(\tilde{b}^\alpha(x)v_\alpha(x)\right)$, and that the bilinear form $\blf^*[\bm v,\bm u]$ associated with \eqref{E:wces-adjoint} satisfies $\blf^*[\bm v,\bm u]=\blf[\bm u,\bm v]$ for all $\bm u,\bm v \in \HH^1$. Solutions to the adjoint systems $(-L^*-Q^*)\bm v=\bm \zeta$ or $(\lambda I_\mm -L^*-Q^*)\bm v=\bm \zeta$ are defined in a manner similar to that in Definition \ref{D:weak-solution}.
\end{remark}

%For $u(x)=(u_1(x),\dots,u_\mm(x)) \in \HH^1$ and $\psi(x)=(\psi_1(x),\dots,\psi_\mm(x)) \in \LL$, the system of {\sc pde} $-\gen u = \psi$ can be written

% Coercivity
\begin{lemma}[Coercivity]\label{L:coercivity}
There exist constants $k_{\ref{L:coercivity}}>0$, $\lambda_{\ref{L:coercivity}} \ge 0$ such that 
\begin{equation*}
k_{\ref{L:coercivity}}\|\bm u\|_{\HH^1}^2 \le \blf[\bm u,\bm u] + \lambda_{\ref{L:coercivity}} \|\bm u\|_{\LL}^2 \qquad \text{for all $\bm u \in \HH^1$.}
\end{equation*}
\end{lemma}

\begin{proof}
From the uniform ellipticity condition \eqref{E:ellipticity} in Assumption \ref{A:DD}, we easily see from \eqref{E:blf} that for $\bm u \in \HH^1$, we have
%\begin{multline*}
$\frac{1}{2}a_\sfmin \sum_{\alpha=1}^\mm \int_{\BT^\dd} |\nabla u_\alpha(x)|^2 \thinspace dx \le \blf[\bm u,\bm u] + J_1 + J_2$, 
%\qquad \text{where}\\
where $J_1 \triangleq \sum_{\alpha=1}^\mm \int_{\BT^\dd} \left(\tilde{b}^\alpha(x),\nabla u_\alpha(x) \right)_{\BR^\dd} u_\alpha(x) \thinspace dx$, $J_2 \triangleq \sum_{\alpha=1}^\mm \sum_{\beta=1}^\mm \int_{\BT^\dd} q_{\alpha\beta}(x) \left[u_\beta(x)-u_\alpha(x)\right] u_\alpha(x) \thinspace dx$.
%\end{multline*}
Set $B \triangleq \max_{1 \le \alpha \le \mm}\|\tilde{b}^\alpha\|_{L^\infty(\BT^\dd)}$. Then, for any $\delta>0$, it follows from Cauchy's inequalities that 
%\begin{multline*}
$|J_1|  \le \sum_{\alpha=1}^\mm \int_{\BT^\dd} |\tilde{b}^\alpha(x)| |\nabla u_\alpha(x)| |u_\alpha(x)| \thinspace dx \\\le B\sum_{\alpha=1}^\mm \left\{\delta\int_{\BT^\dd}|\nabla u_\alpha(x)|^2 \thinspace dx + \frac{1}{4\delta}\int_{\BT^\dd} |u_\alpha(x)|^2 \thinspace dx\right\} = B\delta \sum_{\alpha=1}^\mm \|\nabla u_\alpha\|_{L^2_\per(\BT^\dd)}^2 + \frac{B}{4\delta} \sum_{\alpha=1}^\mm \|u_\alpha\|^2_{L^2_\per(\BT^\dd)}$.
%\end{multline*}
Turning to $J_2$, and recalling $\bar{q}$ from Assumption \ref{A:SI}, we easily see from Cauchy's inequality that\\
%\begin{equation*}
$|J_2| \le \bar{q} \sum_{\alpha=1}^\mm \sum_{\beta=1}^\mm \int_{\BT^\dd} \left\{\frac{1}{2}|u_\alpha(x)|^2 + \frac{1}{2}|u_\beta(x)|^2 + |u_\alpha(x)|^2\right\} \thinspace dx = 2 \bar{q}\mm \sum_{\alpha=1}^\mm \|u_\alpha\|^2_{L^2_\per(\BT^\dd)} = 2\bar{q}\mm \|\bm u\|^2_\LL$. \\
%\end{equation*}
Putting the pieces together, we now have
%\begin{equation*}
$\frac{1}{2}a_\sfmin \sum_{\alpha=1}^\mm \int_{\BT^\dd} |\nabla u_\alpha(x)|^2 \thinspace dx \le \blf[\bm u,\bm u] + B\delta \sum_{\alpha=1}^\mm \|\nabla u_\alpha\|_{L^2_\per(\BT^\dd)}^2 + \left(\frac{B}{4\delta} + 2\bar{q}\mm\right) \|\bm u\|^2_\LL$.
%\end{equation*}
Taking $\delta=\frac{a_\sfmin}{4B}$ and rearranging, we get 
%\begin{equation*}
$\frac{1}{4}a_\sfmin \sum_{\alpha=1}^\mm \int_{\BT^\dd} |\nabla u_\alpha(x)|^2 \thinspace dx \\\le \blf[\bm u,\bm u] +  \left(\frac{B}{4\delta} + 2\bar{q}\mm\right) \|\bm u\|^2_\LL$.
%\end{equation*}
Adding $\frac{1}{4}a_\sfmin \|\bm u\|^2_\LL$ on both sides, we get the stated result with $k_{\ref{L:coercivity}}=a_\sfmin/4$ and $\lambda_{\ref{L:coercivity}}=\frac{B^2}{a_\sfmin} + 2\bar{q} \mm + \frac{a_\sfmin}{4}$.
\end{proof}

% Boundedness
\begin{lemma}[Boundedness]\label{L:boundedness}
There exists a constant $K_{\ref{L:boundedness}}>0$ such that
\begin{equation*}
|\blf[\bm u,\bm v]| \le K_{\ref{L:boundedness}}\|\bm u\|_{\HH^1} \|\bm v\|_{\HH^1} \qquad \text{for all $\bm u,\bm v \in \HH^1$.}
\end{equation*}
\end{lemma}

\begin{proof}
Start by noting that $|\blf[\bm u,\bm v]| \le \sum_{i=1}^3 K_i$, where $K_1 \triangleq \frac{1}{2}\sum_{\alpha=1}^\mm \int_{\BT^\dd} \left|\left(a^\alpha(x)\nabla u_\alpha(x),\nabla v_\alpha(x)\right)_{\BR^\dd}\right| \thinspace dx$,\\ $K_2 \triangleq \sum_{\alpha=1}^\mm \int_{\BT^\dd} \left|\left(\tilde{b}^\alpha(x),\nabla u_\alpha(x)\right)_{\BR^\dd} v_\alpha(x)\right| \thinspace dx$, $K_3 \triangleq \sum_{\alpha=1}^\mm \sum_{\beta=1}^\mm \int_{\BT^\dd} \left|q_{\alpha\beta}(x)\left(u_\beta(x)-u_\alpha(x)\right) v_\alpha(x)\right| \thinspace dx$.
Routine calculations using the Cauchy-Schwarz and Holder inequalities, together with the boundedness of $a^\alpha(x)$, $b^\alpha(x)$, $q_{\alpha\beta}(x)$ yield the stated result. 
%\color{violet} (\textsf{Details in handwritten notes, in case needed.}) \color{black}
\end{proof}

% Lax-Milgram, compactness of resolvent
\begin{lemma}\label{L:Perturbed-PDE}
Let $\lambda \ge \lambda_{\ref{L:coercivity}}$ and let $I_\mm$ denote the $\mm \times \mm$ identity matrix. Then, for each $\bm \psi \in \LL$, there exists a unique weak solution $\bm{u}_\lambda \in \HH^1$ of the perturbed system of {\sc pde}
\begin{equation}\label{E:Perturbed-PDE}
(\lambda I_\mm - L - Q)\bm{u}_\lambda = \bm\psi.
\end{equation}
Further, the resolvent operator $G_\lambda \triangleq (\lambda I_\mm - L - Q)^{-1}:\LL \to \LL$ which maps the inhomogeneity $\bm\psi \in \LL$ to the unique solution $\bm{u}_\lambda \in \HH^1 \subset \LL$ of \eqref{E:Perturbed-PDE} is compact. 
\end{lemma}

\begin{proof}
Fix $\lambda \ge \lambda_{\ref{L:coercivity}}$, and define $\blf_\lambda:\HH^1 \times \HH^1 \to \BR$ by $\blf_\lambda[\bm u,\bm v] \triangleq \blf[\bm u,\bm v] + \lambda (\bm u,\bm v)_\LL$. 
By Lemma \ref{L:coercivity}, we have $k_{\ref{L:coercivity}}\|\bm u\|_{\HH^1}^2 \le \blf_\lambda[\bm u,\bm u]$ for all $\bm u \in \HH^1$. 
Next, we see from Lemma \ref{L:boundedness} that $|\blf_\lambda[\bm u,\bm v]| \le (K_{\ref{L:boundedness}}+\lambda) \|\bm u\|_{\HH^1} \|\bm v\|_{\HH^1}$ for all $\bm u,\bm v \in \HH^1$. 
Note also that for any $\bm\psi \in \LL$ fixed, $\bm v \mapsto (\bm \psi,\bm v)_{\LL}$ defines a bounded linear functional on $\HH^1$.
Hence, by the Lax-Milgram Theorem (in the Hilbert space $\HH^1$), it follows that for each $\bm \psi \in \LL$, there exists a unique $\bm{u}_\lambda \in \HH^1$ satisfying $\blf_\lambda[\bm u,\bm v]=(\bm\psi,\bm v)_\LL$ for all $\bm v \in \HH^1$, i.e., a unique weak solution of \eqref{E:Perturbed-PDE}. In light of this, we can define a linear operator $G_\lambda \triangleq (\lambda I_\mm - L - Q)^{-1}$ which maps the inhomogeneity $\bm\psi \in \LL$ to the unique solution $\bm{u}_\lambda \in \HH^1 \subset \LL$ of \eqref{E:Perturbed-PDE}. 

We now want to show that the operator $G_\lambda:\LL \to \LL$ is compact. The first step is to show that $G_\lambda$ is a bounded linear operator from $\LL$ to $\HH^1$. To see this, let $\bm \psi \in \LL$ and write $\bm u=G_\lambda \bm \psi$. Note that by coercivity, we have $k_{\ref{L:coercivity}}\|\bm u\|_{\HH^1}^2 \le \blf_\lambda[\bm u,\bm u] = (\bm \psi, \bm u)_\LL \le \|\bm \psi\|_\LL \|\bm u\|_\LL$, where the latter follows from the Cauchy-Schwarz inequality. Recalling that $\|\bm u\|_\LL \le \|\bm u\|_{\HH^1}$, we now get $\|G_\lambda \bm \psi\|_{\HH^1} \le (1/k_{\ref{L:coercivity}})\|\bm \psi\|_\LL$, thereby proving the boundedness claim. Now suppose that $\{\bm \psi_n\}_{n=1}^\infty$ is a bounded sequence in $\LL$. By the calculation above, it follows that $\{G_\lambda \bm \psi_n\}_{n=1}^\infty$ is a bounded sequence in $\HH^1$. Since $\HH^1$ is compactly embedded in $\LL$ (\cite{Robinson}, \cite{Temam}), 
%\color{red} (\textsf{Check this!})\color{black}, 
it follows that there exists a subsequence $\{G_\lambda \bm \psi_{n_k}\}_{k=1}^\infty$ and $\bm v \in \LL$ such that $\|G_\lambda \bm \psi_{n_k} - \bm v\|_\LL \to 0$ as $k \to \infty$. In other words, $G_\lambda:\LL \to \LL$ is a compact operator.
\end{proof}

Note that a similar result holds for the adjoint system of {\sc pde} \eqref{E:wces-adjoint}. We now provide the proof of the Fredholm alternative.

% Lax-Milgram, compactness of resolvent for adjoint
%\begin{lemma}\label{L:Perturbed-PDE-adjoint}
%Let $\lambda \ge \lambda_{\ref{L:coercivity}}$ and let $I_\mm$ denote the $\mm \times \mm$ identity matrix. Then, for each $\bm \zeta \in \LL$, there exists a unique weak solution $\bm{v}_\lambda \in \HH^1$ of the perturbed system of {\sc pde}
%\begin{equation}\label{E:Perturbed-PDE-adjoint}
%(\lambda I_\mm - L^* - Q^*)\bm{v}_\lambda = \bm\zeta.
%\end{equation}
%Further, the resolvent operator $H_\lambda \triangleq (\lambda I_\mm - L^* - Q^*)^{-1}:\LL \to \LL$ which maps the inhomogeneity $\bm\zeta \in \LL$ to the unique solution $\bm{v}_\lambda \in \HH^1 \subset \LL$ of \eqref{E:Perturbed-PDE-adjoint} is compact. 
%\end{lemma}
%
%\begin{proof}
%The proof easily follows from the observations in Remark \ref{R:adjoint-pde} using arguments similar to those above. 
%\end{proof}

\begin{proof}[Proof of Proposition \ref{P:Fredholm}]
We start by recalling the notation at the beginning of Section \ref{S:Fredholm} for expressing functions $f(x,\alpha):\BR^\dd \times \II \to \BR$ in terms of vector-valued functions $\bm f=(f_1,\dots,f_\mm):\BR^\dd \to \BR^\mm$ and the corresponding equivalent statements \eqref{E:generic-elliptic-pde} and \eqref{E:wces} of the same {\sc pde} .

Let $\lambda \ge \lambda_{\ref{L:coercivity}}$ and suppose that $\bm \psi \in \LL$. The function $\bm u \in \HH^1$ solves $(-L-Q)\bm u =\bm \psi$ if and only if $(\lambda I_\mm - L - Q)\bm u=\bm \psi + \lambda \bm u$, i.e., $\bm u$ solves $\bm u = G_\lambda(\bm \psi + \lambda \bm u)$. Noting that the right-hand side of the latter belongs to $\HH^1$ whenever $\bm u \in \LL$, we conclude that 
\begin{equation}\label{E:Fredholm-setup}
\text{$\bm u \in \LL$ belongs to $\HH^1$ and solves $(-L-Q)\bm u =\bm \psi$ if and only if $(I_\mm - \lambda G_\lambda) \bm u = G_\lambda \bm \psi$}.
\end{equation} 
Owing to the compactness of the operator $\lambda G_\lambda$, we now use the Fredholm alternative (in the Hilbert space $\LL$) to deduce that either
\begin{enumerate}
\item[(\texttt{A})] For every $\bm h \in \LL$ (and hence, in particular, for $\bm h=G_\lambda \bm \psi$), the equation $(I_\mm-\lambda G_\lambda)\bm u=\bm h$ has a unique solution, or 
\item[(\texttt{B})] The homogeneous equations $(I_\mm - \lambda G_\lambda^*)\bm v_0=0$ and $(I_\mm - \lambda G_\lambda)\bm V_0=0$ have nontrivial solutions with $1 \le \mathsf{dim}(\nullspace(I_\mm - \lambda G_\lambda))=\mathsf{dim}(\nullspace(I_\mm - \lambda G_\lambda^*))< \infty$. In this case, $(I_\mm-\lambda G_\lambda)\bm u=\bm h$ has a solution if and only if $\bm h$ is orthogonal to the null space of $I_\mm - \lambda G_\lambda^*$, i.e., $(\bm h,\bm v_0)_\LL=0$ for every $\bm v_0 \in \nullspace(I_\mm - \lambda G_\lambda^*)$.
\end{enumerate}

We next establish that the second of these possibilities, viz., (\texttt{B}) holds.
To start, we claim that $(I_\mm - \lambda G_\lambda^*)\bm v_0=0$ iff $\bm v_0$ is a weak solution of $(-L^*-Q^*)\bm v_0=0$ and $(I_\mm - \lambda G_\lambda)\bm V_0=0$ iff $\bm V_0$ is a weak solution of $(-L-Q)\bm V_0=0$.
% \color{orange} (\textsf{I want to show that $\bm v_0 \in \nullspace(I_\mm - \lambda G_\lambda^*)$ iff $(-L^*-Q^*)\bm v_0=0$, i.e., $\blf^*[\bm v_0,\bm w]=0$ for all $\bm w \in \HH^1$.}) \color{black}\\
We now show that the solution spaces to the {\sc pde} $(-L-Q)\bm V_0=0$ and $(-L^*-Q^*)\bm v_0=0$ are one-dimensional. Proposition B.2 in \cite{PeletierSchlottke_EJP2024} on principal eigenvalues for certain weakly coupled elliptic systems with smooth periodic coefficients implies that there exists a unique $\lambda \in \BR$ and a smooth vector-valued function $\bm V_0$ with positive components satisfying $(-L-Q)\bm V_0=\lambda \bm V_0$. Recalling Assumption \ref{A:SI}, it easily follows that the vector function $\bm V_0=(1,\dots,1)$ satisfies $(-L-Q)\bm V_0=0$. Now, we use Proposition 1 in \cite{AllaireHutridurga-DCDSB-2015} to conclude that $\lambda = 0$ is also a real, simple eigenvalue of the adjoint system $(-L^*-Q^*)\bm v_0=\lambda \bm v_0$. Further, the eigenfunctions $(v_{0,1},\dots,v_{0,\mm})$ of the adjoint system are positive and unique up to the normalization $\sum_{\alpha=1}^\mm \int_{\BT^\dd} v_{0,\alpha}(x) \thinspace dx=1$. Applying Proposition B.2 in \cite{PeletierSchlottke_EJP2024} to the adjoint system $(-L^*-Q^*)\bm v_0=0$ assures us that the functions $v_{0,\alpha}(x)$, $1 \le \alpha \le \mm$, are also smooth. Now setting
\begin{equation*}
m(x,\alpha) \triangleq v_{0,\alpha}(x) \quad \text{for $x \in \BR^\dd$, $1 \le \alpha \le \mm$,}
\end{equation*}
we obtain the unique invariant density $m(x,\alpha)$ in \eqref{E:inv-dens-pde} solving $\gen^* m(x,\alpha)=0$.

Let $\varphi:\BR^\dd \times \II \to \BR$ be a smooth function which is $1$-periodic in $x$. Note that the {\sc pde} $\gen \Phi(x,\alpha)=\varphi(x,\alpha)$ has a solution iff the system $(-L-Q)\bm \Phi=-\bm \varphi$ has a solution, which, by \eqref{E:Fredholm-setup}, is equivalent to $(I-\lambda G_\lambda)\bm \Phi=-G_\lambda \bm \varphi$. By the Fredholm alternative (\texttt{B}), this occurs precisely when $-G_\lambda \bm \varphi$ is orthogonal to $\nullspace(I_\mm - \lambda G_\lambda^*)$. Recalling the normalized eigenfunction $\bm v_0$ above which spans $\nullspace(I_\mm - \lambda G_\lambda^*)$ and satisfies $(1/\lambda)\bm v_0=G_\lambda^* \bm v_0$, and using properties of adjoints, we see that 
%\begin{equation*}
$(-L-Q)\bm \Phi=-\bm \varphi$ has a solution if and only if $-(G_\lambda \bm \varphi,\bm v_0)_\LL=-(\bm \varphi,G_\lambda^* \bm v_0)_\LL=-\frac{1}{\lambda}(\bm \varphi,\bm v_0)_\LL=0$. In other words, $\gen \Phi(x,\alpha)=\varphi(x,\alpha)$ has a solution if and only if $\sum_{\alpha=1}^\mm \int_{\BT^\dd} \varphi(x,\alpha) m(x,\alpha) \thinspace dx=0$, as claimed. To show smoothness of $\Phi$ when the solvability condition \eqref{E:centering} holds, one appeals to the smoothness of the drift-diffusion coefficients and switching intensities and leans on smoothness results from the {\sc pde} literature; see, for instance \cite{CDN_HAL}. 
%\end{equation*}
\end{proof}

% Final section: Conclusions
\section{Conclusions}\label{S:Conclusions}
In the present work, we have focused on characterizing the limiting fluctuations, about a constant velocity drift, for a multiscale switching diffusion process with purely periodic drift/diffusion coefficients and switching intensities. The novel finding was the identification of an extra contribution to the limiting diffusion matrix stemming from the switching. A key role was played in the analysis by the assumptions of uniformly elliptic diffusion matrices, irreducible switching dynamics, and purely periodic coefficients. It would be interesting to see if one or more of these assumptions can be weakened. One direction would be to explore homogenization for switching diffusions in the hypoelliptic case, building on the work on \cite{HairerPavliotis_JStatPhys2004}. Another line of inquiry would be to consider the case of \textit{weakly periodic} coefficients; in other words, to study systems similar to those considered in \cite{DellaCorteKraaij} in a diffusive scaling. Finally, it would be worth exploring how the results here can be useful in the analysis of flashing ratchet models of molecular motors, the workings of which provided part of the impetus for the efforts in \cite{PeletierSchlottke_EJP2024}, \cite{DellaCorteKraaij}, and the present work.

\bibliographystyle{alpha}
\bibliography{PerHomSwDiff}

%\item Make precise exact continuity and smoothness assumptions on $b_i(x,\alpha)$, $\sigma_{ij}(x,\alpha)$, $q_{\alpha\beta}(x)$. \\
%\color{NavyBlue} Continuity of $\sigma_{ij}(x,\alpha)$, $q_{\alpha\beta}(x)$ is certainly needed in applying the Ergodic Lemma (Lemma \ref{L:ergodicity}) to $\la M^{\eps,k},M^{\eps,l}\ra_t$. Smoothness of all coefficients will likely greatly facilitate (i) going back and forth between divergence and non-divergence forms, and (ii) smoothness of solutions to the cell problem. \color{black}

\end{document}